\documentclass[11pt]{amsart}
\usepackage{latexsym}
\usepackage{amscd,mathdots}
\usepackage{amssymb}
\usepackage[all, cmtip]{xy}
\usepackage{amsmath}
\usepackage[utf8]{inputenc}
\usepackage{amsthm}
\usepackage{amsmath}
\usepackage{xcolor}
\usepackage[a4paper,  margin=1in]{geometry}
\usepackage{epic}\usepackage{pstricks}
\usepackage[all]{xy}
\usepackage{amssymb}
\usepackage{graphicx}
\usepackage{epic,eepic}
\usepackage{epsfig}
\usepackage{float}
\usepackage{extarrows}
\usepackage{soul}
\DeclareMathOperator\id{}

\newtheorem{thm}{Theorem}[section]

\newtheorem{lem}[thm]{Lemma}

\newtheorem{prob}[thm]{Problem}
\theoremstyle{definition}
\newtheorem{dfn}[thm]{Definition}
\theoremstyle{remark}
\newtheorem{rmk}[thm]{Remark}
\newtheorem{example}[thm]{Example}
\makeatletter

\numberwithin{equation}{section}
\numberwithin{figure}{section}

\usepackage[all]{xy}

\makeatother

\newcommand{\overunder}[2]{ 
\!\begin{array}{c} 
\scriptstyle{#1}\\[-.1in] 
-\!\!\!-\!\!\!-\\[-.1in] 
\scriptstyle{#2} 
\end{array} 
\! 
}

\begin{document}
\title{An introduction to supersymmetric cluster algebras}
\author{Li Li}
\address{Oakland University, Rochester, MI-48309, USA}
\email{li2345@oakland.edu}
\author{James Mixco}
\address{Department of Mathematics and Statistics, St. Louis University, St.
Louis, MO-63103, USA} 
\email{james.mixco@slu.edu}
\author{B. Ransingh}
\address{Harish-Chandra Research Institute, India}
\email{biswajitransingh@hri.res.in}
\author{Ashish K. Srivastava}
\address{Department of Mathematics and Statistics, St. Louis University, St.
Louis, MO-63103, USA} 
\email{ashish.srivastava@slu.edu}
\dedicatory{Dedicated to the loving memory of Anupam Srivastava, the budding 7 year old astronomer who left planet earth too soon}

\thanks{The work of the first author is partially supported by NSA grant H98230-16-1-0303. The work of the fourth author is partially supported by a grant from Simons Foundation (grant number 426367).}

\maketitle

\begin{abstract}
In this paper we propose the notion of cluster superalgebras which is a supersymmetric version of the classical cluster algebras introduced by Fomin and Zelevinsky. We show that the symplectic-orthogonal supergroup $SpO(2|1)$ admits a cluster superalgebra structure and as a consequence of this, we deduce that the supercommutative superalgebra generated by all the entries of a superfrieze is a subalgebra of a cluster superalgebra. We also show that the coordinate superalgebra of the super Grassmannian $G(2|0; 4|1)$ of chiral conformal superspace (that is, $(2|0)$ planes inside the superspace $\mathbb C^{4|1}$) is a quotient of a cluster superalgebra.  
\end{abstract}

\bigskip

\bigskip

\noindent 

\section{Motivation behind the notion of cluster superalgebras} 

\bigskip

\noindent The study of cluster algebras was initiated by Fomin and Zelevinsky in 2001 \cite{FZ1}. Cluster algebras are commutative rings with a set of distinguished generators called cluster variables. These algebras are different from usual algebras in the sense that they are not presented at the outset by a complete set of generators and relations. Instead, an initial seed consisting of initial cluster variables and an exchange matrix is given and then using an iterative process called mutation, the rest of the cluster variables are generated. Examples of cluster algebras include the homogeneous coordinate ring $\mathcal A=\mathbb C[SL_2]$ of the group $SL_2$ and the homogeneous coordinate ring of any Grassmannian.

The initial motivation behind study of these algebras was to provide an algebraic and combinatorial framework for Lusztig's work on canonical bases but now the study of cluster algebras has gone far beyond that initial motivation. Cluster algebras now have connections to string theory, Poisson geometry, algebraic geometry, combinatorics, representation theory and Teichmuller theory. Cluster algebras provide a unifying algebraic and combinatorial framework for a wide variety of phenomenon in above mentioned settings. Interestingly, the mutation rule proposed by Fomin and Zelevinsky came up naturally in the Seiberg-Witten duality in string theory \cite{SW}. 

One of the challenges in applying cluster algebra to quiver gauge theory is that it is difficult to characterize when every dual theory obtained by successive applications of Seiberg-like dualities has enough quadratic terms in the superpotential that, after integrating out the chiral multiplets involved, no oriented 2-cycles are left. So, one of our motivations behind the introduction of the notion of cluster superlagberas is to deal with the situation when we have oriented 2-cycles. 

To emphasize the need to be able to deal with quivers having oriented 2-cycles, let us consider an appropriate on-shell diagram representing scattering amplitudes. Note that the study of scattering amplitudes is crucial to our understanding of quantum field theory. Scattering amplitudes are complicated functions of the helicities and momenta
of the external particles. But to visually interpret them in an easier manner, one may label particles involved by $\{1, 2, \ldots, n\}$ and the interaction of particles involved could be associated with a permutation of $\{1, 2, \ldots, n\}$. One of the ways to represent  scattering amplitudes is an on-shell diagram. See Arkani-Hamed et al \cite{Nima} for more details on scattering amplitudes and on-shell diagrams. 

Consider a planar on-shell diagram $B$ given as follows: 
\begin{figure}[h]
 \centering 
\includegraphics[width=4cm]{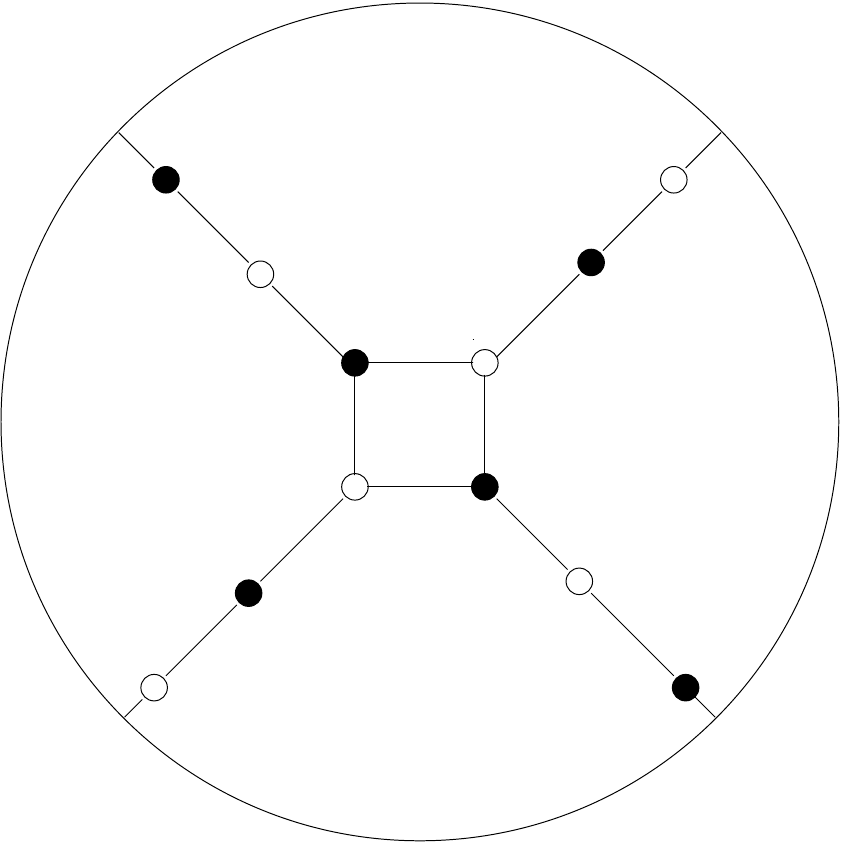}
\end{figure}

\noindent Here black and white vertices represent particles with opposite helicities involved in the scattering. We associate a quiver $Q=Q(B)$ with the bipartite graph $B$ in the following manner.
Take a vertex for each face and for each edge in the graph $B$, we draw an arrow in this quiver in such a way that it sees the white vertex in left. So we get the following quiver:

$$
\begin{array}{ccccccc}
\xymatrix{
&	& &3\ar@{<-}[rrdd]&&&\\
&	&  &&
	&&\\
Q=Q(B):&	4\ar@{<-}[rr]\ar@{->}[uurr]\ar@/^1.0pc/@{<-}[uurr]&& 1\ar@{->}[rr]\ar@{<-}[uu]\ar@{<-}[dd]&&
	  2\ar@/_1.0pc/@{<-}[uull] &&\\
&	&&&&&\\
&	&&5\ar@{<-}[lluu]\ar@{<-}[rruu]\ar@/^1.0pc/@{->}[uull]\ar@/_1.0pc/@{->}[uurr] &&&
}
\end{array}
$$
This quiver has oriented 2-cycles and so we cannot deal with this in the set up of classical cluster algebras. So we need an extension of cluster algebras to study algebras arising from such quivers.  

The study of symmetry has always been the central idea in mathematics and physics. In the theory of spin, we deal with rotational symmetry. Poincare symmetry is crucial in understanding of the classification of elementary particles. Similarly, permutation symmetry plays an important role in dealing with the systems of identical particles. But when one deals with both the bosons and the fermions, then to formulate the symmetry arising in this situation, the ordinary Lie group theory is insufficient. Superalgebras were introduced by physicists to provide an algebraic framework for describing boson-fermion symmetry. The boson-fermion symmetry is called supersymmetry and it holds the key to a unified field theory.  See \cite{V} for more details on supersymmetry.

A super vector space $V$ is a vector space that is $\mathbb Z_2$-graded, that is, it has a decomposition $V=V_0 \oplus V_1$ with $0, 1\in \mathbb Z_2:=\mathbb Z/2\mathbb Z$. The elements of $V_0$ are called the {\it even} (or bosonic) elements and the elements of $V_1$ are called the {\it odd} (or fermionic) elements. The elements in $V_0\cup V_1$ are called {\it homogeneous} and their {\it parity}, denoted by $p$, is defined to be $0$ or $1$ according as they are even or odd. The morphisms in the category of super vector spaces are linear maps which preserve the gradings.      

A superalgebra $A$ is an associative algebra with an identity element (which is necessarily an even element) such that the multiplication map $A\otimes A\rightarrow A$ is a morphism in the category of super vector spaces. This is the same as requiring $p(ab)=p(a)+p(b)$ for any two homogeneous elements $a$ and $b$ in $A$. A superalgebra $A$ is called {\it supercommutative} if $ab=(-1)^{p(a)p(b)}ba$, for all (homogeneous) $a, b\in A$. This means in a supercommutative superalgebra, odd elements anticommute with each other, that is, $ab=-ba$ for any two odd elements $a, b\in A$, whereas even elements commute with any other element (even or odd). We refer the reader to \cite{Kac} and \cite{Manin} for further details on superalgebras.

Recently, Ovsienko \cite{Ovsienko} and Ovsienko-Shapiro \cite{OS} have made an inspiring attempt to define cluster superalgebras, which includes several important examples such as the supergroup $OSp(1|2)$, superfriezes, the extended Somos-4 sequence, etc. In these papers, they consider an extension of a quiver by adding odd vertices and make it an oriented hypergraph. Note that in an oriented hypergraph, an arrow can connect any number of vertices. The main limitation of their approach is the lack of exchange relations for the odd variables. Since they do not define a notion of odd mutation, it is quite clear that such a notion cannot provide natural geometric examples of cluster superalgebras like the coordinate superalgebra of super Grassmannian. This is a major setback to any proposed notion of cluster superalgebras because one expects that in a meaningful notion of cluster superalgebras, the supersymmetric analogues of examples of cluster algebras should turn out to be examples of cluster superalgebras. 

The objective behind this paper is to propose a notion of cluster superalgebras which is a natural supersymmetric analogue of classical cluster algebras and provides some interesting geometric examples. In particular, we show that the coordinate superalgebra of the super Grassmannian $G(2|0; 4|1)$ of $(2|0)$ planes inside the superspace $\mathbb C^{4|1}$ is a quotient of a cluster superalgebra. This super Grassmannian $G(2|0; 4|1)$ is called the chiral conformal superspace. The elements of the coordinate superalgebra of super Grassmannian $G(2|0; 4|1)$ are identified with chiral superfields. It may be noted here that chiral superfields appear naturally in supersymmetric theories and so we expect that our results will have some far reaching consequences. We also show that the supersymmetric analogue of the group $SL_2$, the symplectic-orthogonal supergroup $SpO(2|1)$ admits a cluster superalgebra structure. Surprisingly, although our defintion of cluster superalgebra is quite different from that of \cite{Ovsienko}, we are able to prove some results similar to \cite{Ovsienko} in our setting too. For example, we are able to show that the supercommutative superalgebra generated by all the entries of a superfrieze is a subalgebra of a cluster superalgebra which is the main result of \cite{Ovsienko} and \cite{OS}.   

\bigskip

\section{Recollections from cluster algebras}
 
\bigskip

\noindent In this section we recall the basics of cluster algebras of geometric type. We use the notations as in forthcoming textbook ``Introduction to Cluster Algebras'' by Fomin, Williams, and Zelevinsky \cite[\S3.1]{FWZ}; the readers are refered to \cite{FZ1}, \cite{FZ2}, \cite{FZ3} and \cite{FZ4} for more details on cluster algebras.

Let $r, s$ be positive integers with $r\le s$. We set our ambient field $\mathcal F$ to be the field of rational functions over $\mathbb{Q}$ in $s$ independent variables. Let $\mathbb{T}$ be the $r$-regular tree whose edges are labeled by integers in $\{1,\dots,r\}$ and edge edge is incident on $r$ edges with distinct labels. The notation $t \overunder{k}{} t'$ means that the edge joining $t$ and $t'$ is labeled by $k$.
 Each vertex $t\in\mathbb{T}$ is attached with a seed
seed is pair $(\tilde{X}(t),\tilde{B}(t))$ where $\tilde{X}(t)=(x_1(t),\dots,x_s(t))$ such that $x_1(t),\dots,x_s(t)$ is a free generating set
of $\mathcal F$, and $\tilde{B}(t)$ is an $s\times r$ extended skew-symmetrizable integer matrix, in the sense that the submatrix $B(t)$ of $\tilde{B}(t)$ formed by the top $r$ rows is skew-symmetrizable (a square matrix $B$ is called skew-symmetrizable if there exists a diagonal positive integer matrix $D$ such that $DB$ is skew-symmetric). Following the convention in \cite[\S3.1]{FWZ}, we call $\tilde{X}(t)$ the extended cluster of the seed, whose first $r$ elements $x_1(t),\dots,x_r(t)$ are called the cluster variables of the seed and the remaining $s-r$ elements $x_{r+1}(t),\dots,x_s(t)$ are called the frozen variables. (For $i=r+1,\dots,s$, the frozen variable $x_{i}(t)$ does not depend on $t$, so we can simply denote it by $x_i$.)

Given a seed $(\tilde{X}(t)=\{x_1(t), \ldots, x_s(t)\}, \tilde{B}(t)=[b_{ij}^{(t)}])$ at vertex $t$ and an edge  $t \overunder{k}{} t'$ ($1\le k\le r$), the mutation $\mu_k$ gives a seed $(\tilde{X}(t')=\{x_1(t'), \ldots, x_s(t')\}, \tilde{B}(t)=[b_{ij}^{(t')}])$ at vertex $t'$,  determined by
$$x_{j}(t')=x_{j}(t),\hspace{0.5cm}j\neq k$$
\begin{equation} x_{k}(t')=\frac{1}{x_k(t)}\left[\left(\underset{b_{ik}^{(t)}>0}{\prod}x_{i}(t)^{b_{ik}^{(t)}}\underline{}\right)+
\left(\underset{b_{ik}^{(t)}<0}{\prod}x_{i}(t)^{-b_{ik}^{(t)}}\right)\right]\end{equation}
and $\tilde{B}'$ is determined by
$b_{ij}^{(t')}=-b_{ij}^{(t)}$ if $i=k$ or $j=k$, otherwise $$b_{ij}^{(t')}=b_{ij}^{(t)}+\frac{|b_{ik}^{(t)}|b_{kj}^{(t)}+b_{ik}^{(t)}|b_{kj}^{(t)}|}{2}.$$ 

Now fix an initial seed $(\tilde{X}_o=(x_1,\dots,x_s),\tilde{B}_o)$ and apply all possible finite sequences of mutations from it to get a set of seeds. Let $\mathcal X$ be the set of all cluster variables from these seeds. The cluster algebra $\mathcal A(\tilde{X}_o, \tilde{B}_o)$ is defined to be $\mathbb{Z}[x_{r+1}^\pm,\dots,x_{s}^\pm]$-subalgebra of $\mathcal F$ generated by $\mathcal X$. It is a classical result in \cite{FZ1} that all cluster variables are Laurent polynomials with variables $x_1,\dots,x_r$ and coefficients in $\mathbb{Z}[x_{r+1}^\pm,\dots,x_{s}^\pm]$.

\medskip

In the special case when the top $r$ rows of $\tilde{B}$ is skew-symmetric, $\tilde{B}$ can be encoded by a quiver $Q$ which has no loops and no 2-cycles: let the vertex set of $Q$ be $\{x_1,\dots,x_s\}$;  for any entry $b_{ij}$, there are $|b_{ij}|$ arrows between vertices $x_i$ to $x_j$, all of which go from $x_i$ to $x_j$ if $b_{ij}>0$, or from $x_j$ to $x_i$ if $b_{ij}<0$. Then mutations of $\tilde{B}$ can be suitably rephrased in terms of mutations of $Q$ as follows.

 $\mu_k(Q)$ is obtained from $Q$ by keeping the same vertices and changing arrows by the following rule: 
\begin{itemize}
\item [(i)] Add an arrow $x_{i}\rightarrow x_{j}$ for each distinct path $x_{i}\rightarrow x_{k}\rightarrow x_{j}$.
\item[(ii)] Reverse all arrows incident at $x_{k}$.
\item[(iii)] Delete any 2-cycle produced in the process.
\end{itemize}

\bigskip

\section{Definition of cluster superalgebras}

\bigskip

\noindent In this section, we propose a definition of cluster superalgebras. It is inspired by the pioneering preprint by Ovsienko \cite{Ovsienko} where a version of cluster superalgebras is given. In Ovsienko's approach, oriented hypergraphs are used to replace quivers. Our approach still uses quivers and is, in some sense, more in the spirit of the definition of classical cluster algebras.

Fix an initial seed $(X|Y, Q)$ where $Y=\{y_{1}, \ldots, y_{n}\}$ is the set of Grassmann (or odd) variables that anticommute with each other, and $X=\{x_1, \ldots, x_m\}$ is the set of variables that commute with all variables in $X$ and $Y$. These variables $x_1, \ldots, x_m, y_{1}, \ldots, y_{n}$ are called \textit{initial variables}. Let $Q$ be a  quiver with $m+n$ vertices which
are labeled as $x_1, \ldots, x_m, y_{1}, \ldots, y_{n}$. We will call the vertices labeled with $x_1, \ldots, x_m$ as \textit{even} vertices and
the vertices labeled with $y_{1}, \ldots, y_{n}$ as \textit{odd} vertices. By abuse of notations, we sometimes also denote a non-initial-variable by $x_i$ or $y_j$  instead of $x_i(t)$ or $y_j(t)$ to simply the notation. We will remark in place where confusion may arise.

Just like the classical cluster algebras recalled in \S2, there may be two types of even (odd) variables: \textit{mutable} even (odd) variables and \textit{frozen} even (odd) variables. Mutable variables may be transformed into another variable under a mutation map, whereas the frozen variables remain unaffected under any mutation map.

Just as in the case of classical cluster algebras, we do not allow 2-cycles between two even variables or two odd variables. However, we allow 2-cycles between an even vertex and an odd vertex. We do not allow loops on even vertices but allow loops on odd vertices. We will call such a quiver a \textit{superquiver.}

Throughout this paper, if a superquiver has odd vertices we will restrict ourselves to those superquivers that satisfy at least one of the properties below. The reason behind these restrictions is to ensure Laurent phenomenon for our set up of cluster superalgebras. As we will see in \S\ref{section:limitations and extensions} that the Laurent phenomenon breaks down without these restriction.

\bigskip

\subsection{Restriction on our superquivers} 

\bigskip

\noindent In any superquiver with more than one odd vertex, we assume that, for any path of the form $y_i\to x_k\to y_j$ satisfying  ``$i\neq j$, $x_k$ is mutable, and there are no arrows from $y_i$ to $y_j$'',  \emph{at least one} of the following restrictions hold:

\medskip
{(C1)} $x_k$ is not adjacent to any mutable even vertex. (Two vertices in a quiver are called \emph{adjacent} if there is at least one arrow connecting them in either direction.)

\medskip

{(C2)}  The number of arrows from $y_i\to x_k$ is the same as the number of arrows in the opposite direction; the number of arrows from $x_k$ to $y_j$ is the same as the number of arrows in the opposite direction. 

\medskip

\noindent At first sight these requirements on superquivers might seem too restrictive but as we will see later that a lot of interesting geometric examples of cluster superalgebras arise from superquivers which satisfy these restrictions.  

\bigskip

\subsection{Mutation} \label{subsection:Mutation}

\bigskip

\noindent We first define Fomin-Zelevinsky type mutation for our setting. We define two types of mutations: an even mutation and an odd mutation. We will denote the even mutation in the direction of a mutable even vertex $x_k$ as $\mu_k$ and odd mutation in the direction of a mutable odd vertex $y_k$ as $\eta_k$. In below, $x_i$, $y_j$ are not restricted to initial variables unless otherwise specified.

\subsubsection{Even Mutation}

\noindent We define the even mutation in the direction of vertex $x_k$ as
\begin{equation}
\mu_{k}(x_{1},\cdots,x_{m}, y_{1},\cdots,y_{n},Q)=(\mu_k(x_{1}), \cdots, \mu_k(x_{m}), \mu_k(y_{1}), \cdots, \mu_k(y_{n}), \mu_k(Q))
\end{equation}
where
\begin{equation}\label{evenmutation}
\begin{array}{l}
\mu_{k}(y_{i})=y_{i} \hspace{0.5cm} \mbox{for each $i$}\\
\mu_{k}(x_{i})=x_{i},\hspace{0.5cm} \mbox{for each $i\neq k$}\\
 \mu_{k}(x_{k})=\frac{1}{x_k}\bigg[(-1)^u \big(\underset{x_{i}\rightarrow x_{k}}{\prod}x_{i}\underline{}\big)+(-1)^v \big(\underset{x_{k}\rightarrow x_{j}}{\prod}x_{j}\big)+\underset{\xymatrix@C=0.75em{y_i \ar@/_1pc/[rr]|{\SelectTips{cm}{}\object@{/}}|{}\ar[r]
 		& x_k\ar[r]
 		& y_j}}{\sum}y_iy_j\underline{}\bigg]
\end{array}
\end{equation}
where $u$ is the total number of loops on all odd vertices adjacent to some even vertex $x_i$ with arrow $x_i\to x_k$, $v$ is the total number of loops on all odd vertices adjacent to some even vertex $x_j$ with arrow $x_k\rightarrow x_j$ (more precisely, if there are $a$ arrows $x_i\to x_k$, $b$ arrows between $y_j$ and $x_i$, and $c$ loops on $y_j$ then they contribute $abc$ to $u$; a similar definition applies to $v$). In the last term of $\mu_k(x_k)$, $\xymatrix@C=0.75em{y_i \ar@/_1pc/[rr]|{\SelectTips{cm}{}\object@{/}}|{}\ar[r]
 		& x_k\ar[r]
 		& y_j}$ means that $i\neq j$, no arrows from $y_i$ to $y_j$, and if there are $a$ arrows $y_i\rightarrow x_k$, $b$ arrows $x_k\rightarrow y_j$, then it will contribute $ab y_iy_j$ to the sum.  
\smallskip

\noindent The new superquiver $\mu_k(Q)$ is obtained from $Q$ by modifying vertices in view of the above mentioned exchange rules and changing arrows as follows:

\begin{itemize}
\item [(i)] If there is a path $x_{i}\rightarrow x_{k}\rightarrow x_{j}$, add an arrow $x_{i}\rightarrow x_{j}$ for each distinct path.
\item[(ii)] Reverse all arrows connecting $x_{k}$ to another even vertex.
\item[(iii)] Delete any 2-cycle produced between two even variables in the process.
\end{itemize}
It is easy to check that if (C1) (resp. (C2)) holds for $Q$, it also holds for $\mu_k(Q)$.
\begin{rmk}
(1) In \eqref{evenmutation}, the empty product will be considered to be 1 and the empty sum will be considered to be zero by convention. The same convention is applied throughout the paper. 

(2) The mutation $\mu_k(Q)$ is obtained from $Q$ by applying the classical mutation on the induced subquiver consisting of only even vertices, while keeping other arrows unaffected. 

(3) A priori, $\mu_k(x_k)$ may not be well-defined because the quotient $\frac{1}{x_k}$ may not make sense (if the cluster variable $x_k=x_k(t)$ is a zero-divisor). But as we shall see later in Theorem \ref{thm:Laurent}, the new variables obtained by mutation are well-defined, satisfy Laurent phenomenon in the sense that they are elements in $\mathbb K[x_1^{\pm 1}, \ldots, x_m^{\pm 1}] \otimes \mathbb K[y_1, \ldots, y_n]$, where  $x_i$, $y_j$ are initial variables, and $\mathbb K[y_1, \ldots, y_n]$ is the Grassmann algebra.

(4) The above even mutation of cluster variables is inspired by Ovsienko's preprint \cite[v1, Definition 2.4.1]{Ovsienko} which is given below for the reader's convenience:
$$
 \mu_{k}(x_{k})=\frac{1}{x_k}\Big( \underset{x_{i}\rightarrow x_{k}}{\prod}x_{i}\underline{}+\underset{x_{k}\rightarrow x_{j}}{\prod}x_{j}+(\sum_{y_i\to x_k}y_i)(\sum_{x_k\to y_j}y_j)(\prod_{x_i\to x_k} x_i)\Big)
$$
 which also has three terms in the numerator. In that preprint, the mutation is defined assuming the extended quiver to be a hypergraph. 

\end{rmk}

\bigskip

\subsubsection{Odd Mutation}

\noindent We define the odd mutation in the direction of an exchangeable odd vertex $y_i$ as

\begin{equation}
\eta_{i}(x_{1},\cdots,x_{m}, y_{1},\cdots,y_{n}, Q)=(\eta_i(x_{1}), \cdots, \eta_i(x_{m}), \eta_i(y_{1}), \cdots, \eta_i(y_{n}), \eta_i(Q))
\end{equation}

where

\begin{equation}\label{oddmutation}
\begin{array}{l}
\eta_{i}(y_{j})=y_{j},\hspace{0.5cm}j\neq i\\

\eta_{i}(y_{i})=\delta(y_i)y_i+\left(\underset{x_{k}\leftrightarrows y_i}{\prod} \frac{1}{x_k}\right)\left[\left(\underset{y_{i}\rightarrow y_{j}}{\sum}y_{j}\underline{}\right) \left(\underset{\xymatrix@C=0.75em{y_i \ar@/_1pc/[rr]\ar[r]
 & x_l\ar[r]
 & y_j}}{\prod}x_{l}\underline{}\right) +\left(\underset{y_{j}\rightarrow y_{i}}{\sum}y_{j}\underline{}\right) \left(\underset{\xymatrix@C=0.75em{y_j \ar@/_1pc/[rr]\ar[r]
 & x_l\ar[r]
 & y_i}}{\prod}x_{l}\underline{}\right)\right]\\
\eta_{i}(x_{k})=x_{k} \hspace{0.5cm} \mbox{for each $k$}
\end{array}
\end{equation}
where we define $\delta(y_i)=1$ if there is no arrow between $y_i$ and another odd vertex and $\delta(y_i)=0$ otherwise.
We count the multiplicity of ${x_{k}\leftrightarrows y_i}$ to be $ab$ if there are $a$ arrows $x_k\to y_i$ and $b$ arrows $y_i\to x_k$; 
we count the multiplicity of  $ {\xymatrix@C=0.75em{y_i \ar@/_1pc/[rr]\ar[r]
 & x_l\ar[r] & y_j}}$ to be $abc$ if there are $a$ arrows $y_i\to x_l$, $b$ arrows $x_l\to y_j$, $c$ arrows $y_i\to y_j$. 
\noindent  In the expressions $\left(\underset{y_{i}\rightarrow y_{j}}{\sum}y_{j}\underline{}\right)$, $\left(\underset{y_{j}\rightarrow y_{i}}{\sum}y_{j}\underline{}\right)$ and $\left(\underset{\xymatrix@C=0.75em{y_j \ar@/_1pc/[rr]\ar[r]
 & x_l\ar[r]
 & y_i}}{\prod}x_{l}\underline{}\right)$ above, we consider them only for $i\neq j.$ 

\bigskip

\noindent The new superquiver $\eta_i(Q)$ is obtained from $Q$ by modifying vertices in view of the above mentioned exchange rules and changing arrows as follows:

\bigskip

\begin{itemize}
\item [(i)] If there is a path $y_{k}\rightarrow y_{i}\rightarrow y_{j}$ for $k\neq i\neq j$, then add an arrow $y_{k}\rightarrow y_{j}$ for each distinct path.
\item[(ii)] Reverse all arrows incident on $y_{i}$ (including arrows incident on even variables).
\item[(iii)] Delete any 2-cycle produced between two odd variables in the process.
\end{itemize}

\bigskip

\begin{rmk}
(a) The even mutation is involutive, that is, $\mu^2_{k}=1$. 
Indeed, by definition \eqref{evenmutation}, 
$\mu_k(x_k)x_k=A_1+A_2+A_3$, where 
$$A_1=(-1)^u \big(\underset{x_{i}\rightarrow x_{k}}{\prod}x_{i}\underline{}\big),\quad 
A_2=(-1)^v \big(\underset{x_{k}\rightarrow x_{j}}{\prod}x_{j}\big), \quad
A_3=\underset{\xymatrix@C=0.75em{y_i \ar@/_1pc/[rr]|{\SelectTips{cm}{}\object@{/}}|{}\ar[r]
 		& x_k\ar[r]
 		& y_j}}{\sum}y_iy_j\underline{}$$
Similarly, $\mu^2_k(x_k)\mu_k(x_k)=A_2+A_1+A_3$. So $\mu^2_k(x_k)=x_k$. Obviously, $\mu^2_k$ fixes other variables. So $\mu_k^2=1$. 	

It follows that if there are $m'$ number of mutable even variables, then the exchange pattern for even vertices is an $m'$-regular tree. 

(b) The odd  mutation is not involutive in general. Nevertheless, we have $\eta^3_{i}=\eta_i$. To see this, we consider two cases. 

Case 1: if $\delta(y_i)=1$. Then there are no arrows $y_i\to y_j$ or $y_j\to y_i$, therefore $\eta_i(y_i)=y_i$. Since $\eta_i$ fixes other variables, we have $\eta_i=1$, thus $\eta_i^3=\eta_i$. 

Case 2: if $\delta(y_i)=0$. Then
$$
\eta_{i}(y_{i})=\left(\underset{x_{k}\leftrightarrows y_i}{\prod} \frac{1}{x_k}\right)\left[\left(\underset{y_{i}\rightarrow y_{j}}{\sum}y_{j}\underline{}\right) \left(\underset{\xymatrix@C=0.75em{y_i \ar@/_1pc/[rr]\ar[r]
 & x_l\ar[r]
 & y_j}}{\prod}x_{l}\underline{}\right) +\left(\underset{y_{j}\rightarrow y_{i}}{\sum}y_{j}\underline{}\right) \left(\underset{\xymatrix@C=0.75em{y_j \ar@/_1pc/[rr]\ar[r]
 & x_l\ar[r]
 & y_i}}{\prod}x_{l}\underline{}\right)\right]\\
$$
$$
\eta^2_{i}(y_{i})=\left(\underset{x_{k}\leftrightarrows y_i}{\prod} \frac{1}{x_k}\right)\left[\left(\underset{y_{i}\leftarrow y_{j}}{\sum}y_{j}\underline{}\right) \left(\underset{\xymatrix@C=0.75em{y_i 
 & x_l\ar[l]\ar[r]
 & y_j\ar@/^1pc/[ll]}}{\prod}x_{l}\underline{}\right) +\left(\underset{y_{j}\leftarrow y_{i}}{\sum}y_{j}\underline{}\right) \left(\underset{\xymatrix@C=0.75em{y_j \ar[r]
 & x_l
 & y_i\ar[l]\ar@/^1pc/[ll]}}{\prod}x_{l}\underline{}\right)\right]\\
$$
Note that
$\eta^3_{i}(y_{i})=\eta_{i}(y_{i})$  because $\eta^2(Q)=Q$, and $\eta_i(y_i)$ is determined by $Q$ and all variables except $y_i$. This completes the proof that $\eta_i^3=\eta_i$.

\end{rmk}

\bigskip

Now, we proceed to define cluster superalgebras.

\begin{dfn}
Let $\mathcal X_{even}$ be the set of all supercluster even variables that can be obtained by applying a sequence of even mutations to the initial seed $(X|Y, Q)$ and $\mathcal X_{odd}$ be the set of all supercluster odd variables that can be obtained by applying a sequence of odd mutations to the initial seed $(X|Y, Q)$. Then the cluster superalgebra $\mathcal C_{\mathbb K}(X|Y, Q)$ over a field $\mathbb K$ (of characteristic different from $2$) is defined to be the supercommutative $\mathbb K$-superalgebra generated by $\mathcal X_{even} \cup \mathcal X_{odd}$.   
\end{dfn}

\begin{rmk} \rm
If there are no Grassmann variables, then even mutation is exactly the same as Fomin-Zelevinsky mutation for classical cluster algebras.      
\end{rmk}

\begin{rmk}
Implicitly in the definition of cluster superalgebras, a sequence of mutations consisting of both even and odd mutations is not allowed when constructing a cluster superalgebra. However, a sequence of mutations consisting of both even and odd mutations is allowed in the mutation of a superquiver.
\end{rmk}

\bigskip

\subsection{Laurent Phenomenon of Cluster Superalgebras}

\bigskip

\noindent In the case of classical cluster algebras, Fomin and Zelevinsky proved the Laurent phenomenon of cluster algebras. For a skew-symmetrizable cluster algebra, it asserts that any cluster variable $x_k(t)$ can be written as $$u=\frac{f(x_1, \ldots, x_n)}{x_1^{d_1}\cdots x_n^{d_n}},$$ where $x_1,\dots,x_n$ are initial cluster variables, $f$ is a polynomial in $x_1, \ldots, x_n$ with coefficients in $\mathbb Z[x_{n+1}^{\pm},\dots,x_m^{\pm}]$ and $d_i\in \mathbb Z$. It has been recently proved in \cite{GHKK} and \cite{LR} that the coefficients in these Laurent polynomials are positive. 

In this subsection, we show that the supersymmetric analogue of the Laurent phenomenon holds.

Let $n$ be a positive integer, $\mathbb K$ be a field, $B(t_0)=[b_{ij}]$ be a skew-symmetrizable $n\times n$ matrix. Let $\mathbb{K}(z_1,\dots,z_n)$ be a field of rational functions with $n$ independent variables. With the initial seed  $(z_1,\dots,z_n,B(t_0))$, we can define the cluster variables $(z_1(t),\dots,z_n(t))$ at each vertex $t\in\mathbb{T}$ in the classical way, that is, fixing other cluster variables while replacing the variable with index $k$ by a new cluster variable
	\begin{equation}
	\mu_k(z_{k}(t))=\frac{1}{z_k(t)}\left[\left(\underset{b_{ik}^{(t)}>0}{\prod}z_{i}(t)^{b_{ik}^{(t)}}\underline{}\right)+
	\left(\underset{b_{ik}^{(t)}<0}{\prod}z_{i}(t)^{-b_{ik}^{(t)}}\right)\right]
	\end{equation}
with the mutation of $B$ defined as usual.

Now, define a commutative algebra $$S=\mathbb{K}[\varepsilon_1,\ldots,\varepsilon_n]/(\varepsilon_1^2-1,\ldots,\varepsilon_n^2-1).$$
Note that $S$ contains $\mathbb{K}$ as its subalgebra, The ``formal'' elements $\epsilon_i$ are not elements in $\mathbb{K}$, but we can define maps $S\to \mathbb{K}$ by setting $\epsilon_i \mapsto \pm1$, which would degenerate $S$ to $\mathbb{K}$.
We use a modified formula to define ``cluster variables'' as elements in the total quotient ring of $S[x_1,\dots,x_n]$:
	 \begin{equation}\label{eq: modified mutation}
	\mu_k(x_{k}(t))=\frac{1}{x_k(t)}\left[\left(\underset{b_{ik}^{(t)}>0}{\prod}(\varepsilon_ix_{i}(t))^{b_{ik}^{(t)}}\underline{}\right)+
	\left(\underset{b_{ik}^{(t)}<0}{\prod}(\varepsilon_ix_{i}(t))^{-b_{ik}^{(t)}}\right)\right]
\end{equation}
	and with $z_i=\varepsilon_i x_i$ for every $i$. A priori, \eqref{eq: modified mutation} may not be well-defined because $x_k(t)$ may be a zero-devisor in $S[x_1,\dots,x_n]$ (so may not be invertible in the total quotient ring of $S[x_1,\dots,x_n]$). But we have the following lemma to guarantee that it is well-defined. Note that the rings in lemma are commutative.

\begin{lem}\label{lem:phi z}
For each $1\le k\le n$, $x_k(t)$ is not a zero-divisor. The ring homomorphism $\varphi: \mathbb{K}[z_1^{\pm},\dots,z_n^{\pm}]\to S[x_1^{\pm},\dots,x_n^{\pm}]$, $z_i\mapsto \varepsilon_ix_i$ satisfies 
 $$\varphi(z_i(t))=\varepsilon_i x_i(t), \quad \textrm{ for all $i=1,\dots, n$ and all vertices $t\in\mathbb{T}$}.$$
\end{lem}

\begin{proof}
We prove by induction. The statement is obviously true in the base case when $t=t_0$.  We proceed by induction on the distance between $t$ and $t_0$ in $\mathbb{T}$. 
	\begin{equation*}
	\xymatrix{ 
		{ t_0}\ar@{-}[rr]&&
		{ \bullet}\ar@{-}[rr]&&
		{ \bullet}&
		\cdots &
		{ \bullet}\ar@{-}[rr]&&
		{ t}\ar@{-}^{k}[rr]&&
		{ t'}&&\\	
	}
	\end{equation*}
	We have that $z_k(t')=\mu_k(z_{k}(t))=\frac{1}{z_k(t)}\left[\left(\underset{b_{ik}^{(t)}>0}{\prod}z_{i}(t)^{b_{ik}^{(t)}}\underline{}\right)+
	\left(\underset{b_{ik}^{(t)}<0}{\prod}z_{i}(t)^{-b_{ik}^{(t)}}\right)\right].$ 
Now since by assumption $x_k(t)$ is not a zero-divisor, $\varepsilon_i^2=1$, and $z_k(t')$ is in $\mathbb{K}[z_1^{\pm},\dots,z_n^{\pm}]$, we have 
$$\varphi(z_k(t'))= \frac{1}{\varepsilon_kx_k(t)}\left[\left(\underset{b_{ik}^{(t)}>0}{\prod}(\varepsilon_ix_{i}(t))^{b_{ik}^{(t)}}\underline{}\right)+
	\left(\underset{b_{ik}^{(t)}<0}{\prod}(\varepsilon_ix_{i}(t))^{-b_{ik}^{(t)}}\right)\right]=\frac{1}{\varepsilon_k}(x_k(t'))=\varepsilon_k x_k(t').$$

Next, we shall show that $x_k(t')$ is not a zero-divisor in $S[x_1^{\pm},\dots,x_n^{\pm}]$. This follows from McCoy's theorem which, after easy modification, can be stated as: if a Laurent polynomial $f$ is a zero-divisor in $S[x_1^{\pm},\dots,x_n^{\pm}]$, then $rf=0$ for some nonzero $r\in S$. In our situation, denote the Laurent expansion $z_k(t')=\sum a_{i_1\cdots i_n}z_1^{i_1}\cdots z_n^{i_n}$. Then we have a Laurent expansion for $x_k(t')$:
$$x_k(t')=\varepsilon_k\varphi(z_k(t'))=\varepsilon_k\sum a_{i_1\cdots i_n}\varepsilon_1^{i_1}\cdots \varepsilon_n^{i_n}z_1^{i_1}\cdots z_n^{i_n}, \quad a_{i_1\cdots i_n}\in\mathbb{K}.$$ 
If there is a nonzero $r\in S$ such that $rx_k(t')=0$, then $\varepsilon_ka_{i_1\cdots i_n}\varepsilon_1^{i_1}\cdots \varepsilon_n^{i_n}=0$ for all $i_1,\dots,i_n\in\mathbb{Z}$. Since $\varepsilon_i$ are units, we must have $a_{i_1\cdots i_n}=0$ for all $i_1,\dots,i_n\in\mathbb{Z}$, which implies $z_k(t')=0$, a contradiction. Therefore $x_k(t')$ is not a zero-divisor.
\end{proof}
As a consequence, we have 
	\begin{lem}\label{lem:mu_k(x_k) epsilon}
	Let $x_1,\ldots,x_n$ be initial even cluster variables and $\varepsilon_i\in\{1,-1\}$. For a skew-symmetrizable matrix $B=[b_{ij}],$ define mutation as
	 \begin{equation}
	\mu_k(x_{k})=\frac{1}{x_k}\left[\left(\underset{b_{ik}>0}{\prod}(\varepsilon_ix_{i})^{b_{ik}}\underline{}\right)+
	\left(\underset{b_{ik}<0}{\prod}(\varepsilon_ix_{i})^{-b_{ik}}\right)\right]
\end{equation} 
where the mutation of $B$ is defined as usual. Then $x_k(t)$ after any sequence of mutations as defined above is a Laurent polynomial in the initial even cluster variables.
	\end{lem}
	
	\begin{proof}
It follows immediately from Lemma \ref{lem:phi z}.
	\end{proof}
	
Now we are ready to prove the Laurent phenomenon for cluster superalgebras 	with the restriction C1 or C2 on superquivers.

\begin{thm}\label{thm:Laurent}
Let $(X|Y,Q)$ be the initial seed where $X=\{x_1,\dots,x_m\}$, $Y=\{y_1,\dots,y_n\}$, and $Q$ satisfies the restriction on superquivers given before \S\ref{subsection:Mutation}. Then,

(a) after any sequence of iterated even mutations, $x'_k:=\mu_{i_k}\circ \cdots \circ\mu_{i_1}(x_{k})$ is a Laurent polynomial in the initial supercluster variables $x_1,\dots,x_m,y_1,\dots,y_n$. 

(b) after any sequence of iterated odd mutations, $y'_k:=\eta_{i_k}\circ \cdots \circ\eta_{i_1}(y_{k})$ is a Laurent polynomial in the initial supercluster variables $x_1,\dots,x_m,y_1,\dots,y_n$.

As a consequence, all the cluster variables are well-defined, and the cluster superalgebra is a subalgebra of $\mathbb K[x_1^{\pm 1}, \ldots, x_m^{\pm 1}] \otimes \mathbb K[y_1, \ldots, y_n]$ (where $\mathbb K[y_1, \ldots, y_n]$ is the Grassmann algebra).
\end{thm}

\begin{proof}
(a) Since a frozen or odd variable will remain unaffected by iterated even mutations, we only need to consider the case when $x_k$ is even mutable. 

Assume $x_k$ is even mutable. Let $Q_X$ be the induced subquiver of $Q$ with vertex set $X$, and let $Q_{X,k}$ be the connected component (in the sense of the underlying undirected graph) of $Q_X$ containing $x_k$. Let $Q'_{X,k}$ be the induced subquiver of $Q$ whose vertex set is 
$$\{\textrm{vertices in }Q_{X,k}\}\cup\{\textrm{vertices in $Q$ that are  adjacent to vertices in }Q_{X,k}\}.$$ 
So without loss of generality we assume that $Q'_{X,k}=Q_k$ (that is, in the underlying undirected graph of $Q$, for any two even mutable vertices $x_i$ and $x_j$ there is a connected path between them whose vertices are all even mutable). Consider the following two cases.

Case 1: there is only one even mutable vertex, say $x_k$. The new variable $\mu_k(x_k)$ is obviously a Laurent polynomial in the initial supercluster variables. Since $\mu_k$ is involutive, there is no other new variable. So we are done in this case.

Case 2: there are at least two even mutable vertices. Since (C1) does not hold,  (C2) must hold. 
The restriction (C2) on superquiver ensures that the number of arrows $y_i\rightarrow x_k$ is the same as the number of arrows  $x_k\rightarrow y_i$. Similarly, the number of arrows $x_k\rightarrow y_j$ is the same as number of arrows $y_j\rightarrow x_k$. Therefore if there are $b$ arrows $y_i\rightarrow x_k$ (and hence $b$ arrows $x_k\rightarrow y_i$) and $c$ arrows $x_k\rightarrow y_j$ (and hence $c$ arrows $y_j\rightarrow x_k$), then we will have a term $bcy_iy_j +bcy_jy_i$ in the summation of the last formula of (3.2). But as $y_iy_j=-y_jy_i$, this term $bcy_iy_j +bcy_jy_i$ reduces to zero. Therefore,  the summation of the last formula of (3.2) is 0, and the last formula in \eqref{evenmutation} reduces to 
$$ \mu_{k}(x_{k})=\frac{1}{x_k}\bigg[(-1)^u \big(\underset{x_{i}\rightarrow x_{k}}{\prod}x_{i}\underline{}\big)+(-1)^v \big(\underset{x_{k}\rightarrow x_{j}}{\prod}x_{j}\big)\bigg]=\frac{1}{x_k}\left[\left(\underset{b_{ik}>0}{\prod}(\varepsilon_ix_{i})^{b_{ik}}\underline{}\right)+
	\left(\underset{b_{ik}<0}{\prod}(\varepsilon_ix_{i})^{-b_{ik}}\right)\right]
$$
where $\varepsilon_i$ is $+1$ if the total number of loops on all vertices adjacent to $x_i$ is even, or $-1$ otherwise (here we count the multiplicity as follows: if there are  $b$ arrows between $y_j$ and $x_i$ and $c$ loops on $y_j$ then they contribute $bc$ to the total number). Then the conclusion follows from Lemma \ref{lem:mu_k(x_k) epsilon}.

(b) Since the denominator in the formula of $\eta_i(y_i)$ is a product of some initial even variable $x_k$, and the initial even variables remain unchanged under odd mutations, $y_k'$ is computed by an expression whose denominator is a product of some initial even variable $x_k$, and numerator is a polynomial of initial even variables and (the possibly non-initial) odd variables of the form $\eta_{i_{k-1}}\circ \cdots \circ\eta_{i_1}(y_{j})$. By induction, the latter is Laurent polynomial in $X\cup Y$ whose denominator is a product of some initial even variables. Therefore $y_k'$ is a quotient whose denominator is a product of some initial even variable in $X$ and numerator is a polynomial in $X\cup Y$. This proves the assertion.
\end{proof}

\bigskip

To illustrate the above lemmas and theorem, we consider the following examples. 

\begin{example}
(a) 
		\begin{equation*} Q_1=
	\xymatrix{ 
	&&	 y_1\ar@{->}[d]\ar@(ur,ul)_{\id}&&
		 y_2\ar@{->}[d]\ar@(ur,ul)_{\id}\\
		 x_1\ar@{->}[rr] && x_2\ar@{<-}[rr]&&x_3\\
	}
	\end{equation*}\\
	
		\begin{equation*}Q_2=
	\xymatrix{ 
		z_1\ar@{->}[rr] && z_2\ar@{<-}[rr]&&z_3\\
	}
	\end{equation*}
	
	\bigskip
	
\noindent 	Here, set $(\varepsilon_1,\varepsilon_2,\varepsilon_3)=(1,-1,-1).$ On the superquiver $Q_1$,  there is no path of the form $y_i\to x_k\to y_j$ satisfying  ``$i\neq j$, $x_k$ is mutable, and there are no arrows from $y_i$ to $y_j$'', 
so the conditions (C1) and (C2) are vacuously satisfied. We apply rules for even mutation. On quiver $Q_2$ we apply mutation as in classical cluster algebras. We can see that $$\mu_1(x_1)=\frac{1-x_2}{x_1}, \  \mu_2\circ\mu_1(x_2)=\frac{1-x_2-x_1x_3}{x_1x_2}$$ and $\mu_1\circ\mu_2\circ\mu_1(x_1)=\frac{x_1x_3-1}{x_2}.$ 
Also, $$\mu_1(z_1)=\frac{1+z_2}{z_1}, \  \mu_2\circ\mu_1(z_2)=\frac{1+z_2+z_1z_3}{z_1z_2}$$ and $\mu_1\circ\mu_2\circ\mu_1(z_1)=\frac{z_1z_3+1}{z_2}.$
One can verify by computation that $x_k(t)=\varepsilon_kz_k(t)|_{z_i\mapsto \varepsilon_ix_i}.$

\medskip

(b) Here we give an example where the superquiver $Q$  satisfies the condition (C2) but not (C1). The initial seed is $(X|Y)$, where $X=\{x_1,x_2,x_3\}$, $Y=\{y_1,y_2,y_3,y_4\}$.

		\begin{equation*} Q=
	\xymatrix{ 
	         y_2\ar@/^.3pc/@{->}[dr]\ar@/_.3pc/@{<-}[dr] &	 y_1\ar@/^.2pc/@{->}[d]\ar@/_.2pc/@{<-}[d]\ar@{<-}[l]&
		 y_3\ar@/^.3pc/@{->}[dl]\ar@/_.3pc/@{<-}[dl] \ar@{<-}[l]\\
		 x_2\ar@{->}[r] & x_1\ar@{->}[r]&x_3\\
		 & y_4\ar@/^1pc/@{<-}[u]\ar@/^.2pc/@{<-}[u]\ar@/_1pc/@{->}[u]\ar@/_.2pc/@{->}[u]
	}
	\end{equation*}
We have
$$\aligned
\mu_1(x_1)
&= \frac{1}{x_1}\bigg[\big(\underset{x_{i}\rightarrow x_{1}}{\prod}x_{i}\underline{}\big)+ \big(\underset{x_{1}\rightarrow x_{j}}{\prod}x_{j}\big)+\underset{\xymatrix@C=0.75em{y_i \ar@/_1pc/[rr]|{\SelectTips{cm}{}\object@{/}}|{}\ar[r]
 		& x_1\ar[r]
 		& y_j}}{\sum}y_iy_j\underline{}\bigg]\\
&= \frac{1}{x_1}\bigg[(x_2)+ (x_3)+(2y_2y_4+2y_1y_4+2y_3y_4+2y_4y_2+2y_4y_1+2y_4y_3)\bigg]
= \frac{x_2+x_3}{x_1}\\		
\endaligned		
$$
Let us also compute a sequence of odd mutations:

$\eta_1(y_1)=\frac{1}{x_1}[y_3x_1+y_2x_1]=y_2+y_3$, and
		\begin{equation*} \eta_1(Q)=
	\xymatrix{ 
	         y_2\ar@/^.3pc/@{->}[dr]\ar@/_.3pc/@{<-}[dr] \ar@/^1pc/@{->}[rr]&	 y_1\ar@/^.2pc/@{->}[d]\ar@/_.2pc/@{<-}[d]\ar@{->}[l]&
		 y_3\ar@/^.3pc/@{->}[dl]\ar@/_.3pc/@{<-}[dl] \ar@{->}[l]\\
		 x_2\ar@{->}[r] & x_1\ar@{->}[r]&x_3\\
		 & y_4\ar@/^1pc/@{<-}[u]\ar@/^.2pc/@{<-}[u]\ar@/_1pc/@{->}[u]\ar@/_.2pc/@{->}[u]
	}
	\end{equation*}

 $\eta_2\eta_1(y_2)=\frac{1}{x_1}[y_3x_1+\eta_1(y_1)x_1]=y_2+\eta_1(y_1)=2y_2+y_3$, and
		\begin{equation*} \eta_2\eta_1(Q)=
	\xymatrix{ 
	         y_2\ar@/^.3pc/@{->}[dr]\ar@/_.3pc/@{<-}[dr] \ar@/^1pc/@{<-}[rr]&	 y_1\ar@/^.2pc/@{->}[d]\ar@/_.2pc/@{<-}[d]\ar@{<-}[l]&
		 y_3\ar@/^.3pc/@{->}[dl]\ar@/_.3pc/@{<-}[dl]\\
		 x_2\ar@{->}[r] & x_1\ar@{->}[r]&x_3\\
		 & y_4\ar@/^1pc/@{<-}[u]\ar@/^.2pc/@{<-}[u]\ar@/_1pc/@{->}[u]\ar@/_.2pc/@{->}[u]
	}
	\end{equation*}
	
 $\eta_3\eta_2\eta_1(y_3)=\frac{1}{x_1}[\eta_2\eta_1(y_2)x_1+0]=2y_2+y_3$.

\end{example}

\bigskip

\section{Combinatorial geometric model of even and odd mutations}

\bigskip

\noindent It is well known that flip along a diagonal in the triangulation of a regular polygon provides a geometric model for mutation in the case of cluster algebras. In the similar spirit, we would like to provide a combinatorial geometric model for odd mutations in the case of cluster superalgebras. As we have already noted that if there are no Grassmann variables, then even mutation is exactly the same as Fomin-Zelevinsky mutation for classical cluster algebras. Consider a planar on-shell diagram $B$ given as follows: 
\begin{figure}[h]
 \centering 
\includegraphics[width=4cm]{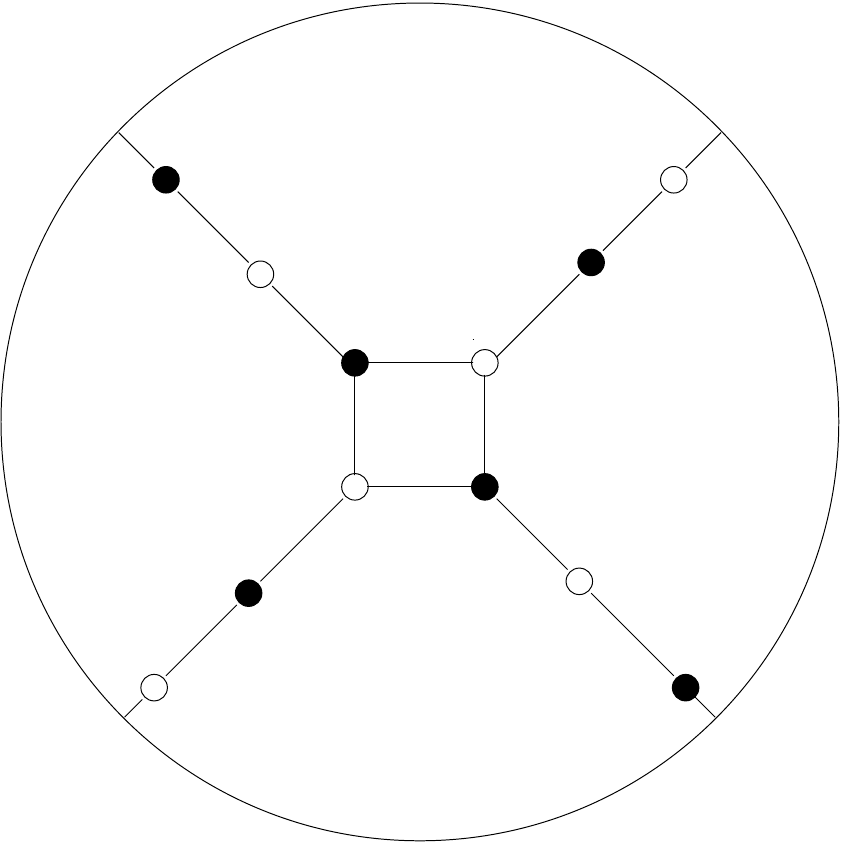}
\end{figure}

\noindent We associate a quiver $Q=Q(B)$ with the bipartite graph $B$ as discussed in the introduction.
We consider a vertex for each face and for each edge in the graph $B$, we draw an arrow in this quiver in such a way that it sees the white vertex in left. So we get the following quiver:

$$
\begin{array}{ccccccc}
\xymatrix{
&	& &3\ar@{<-}[rrdd]&&&\\
&	&  &&
	&&\\
Q=Q(B):&	4\ar@{<-}[rr]\ar@{->}[uurr]\ar@/^1.0pc/@{<-}[uurr]&& 1\ar@{->}[rr]\ar@{<-}[uu]\ar@{<-}[dd]&&
	  2\ar@/_1.0pc/@{<-}[uull] &&\\
&	&&&&&\\
&	&&5\ar@{<-}[lluu]\ar@{<-}[rruu]\ar@/^1.0pc/@{->}[uull]\ar@/_1.0pc/@{->}[uurr] &&&
}
\end{array}
$$

\noindent If we do the ``flip" move on $B$ (which swaps white and black vertices), we get the new planar on-shell diagram $B'$ as 
\begin{figure}[H]
 \centering
 \includegraphics[width=5cm]{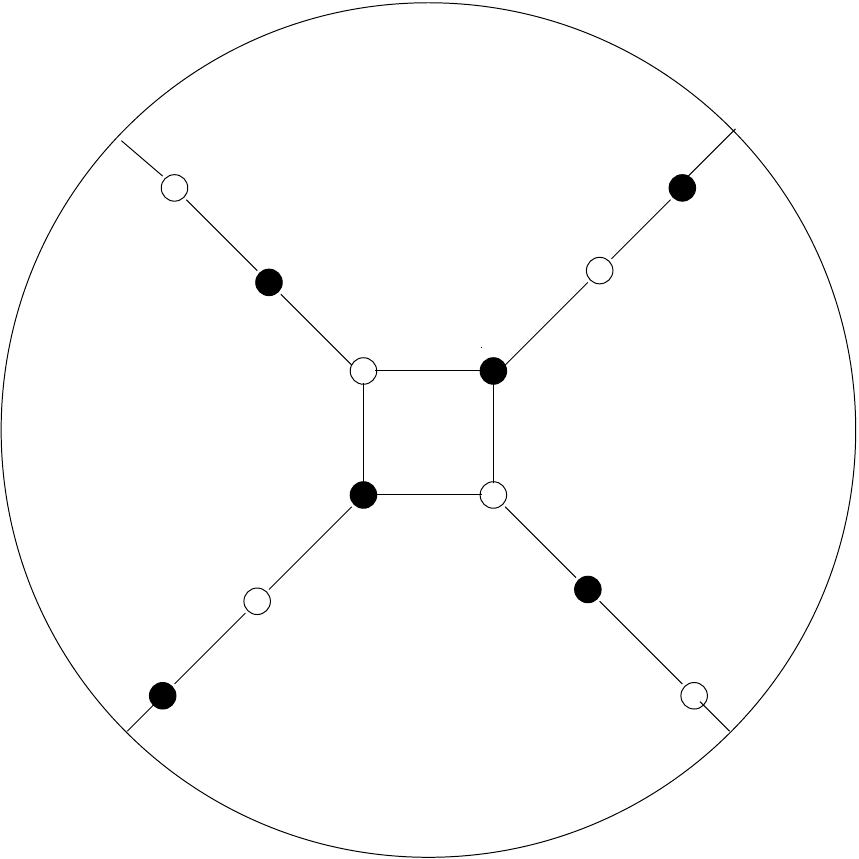}
\end{figure}

Clearly the quiver $Q'=Q(B')$ associated to this bipartite graph is 

$$
\begin{array}{ccccccc}
\xymatrix{
	&	& &3\ar@{->}[rrdd]&&&\\
	&	&  &&
	&&\\
	Q'=Q(B'):&	4\ar@{->}[rr]\ar@{<-}[uurr]\ar@/^1.0pc/@{->}[uurr]&& 1\ar@{<-}[rr]\ar@{->}[uu]\ar@{->}[dd]&&
	2\ar@/_1.0pc/@{->}[uull] &&\\
	&	&&&&&\\
	&	&&5\ar@{->}[lluu]\ar@{->}[rruu]\ar@/^1.0pc/@{<-}[uull]\ar@/_1.0pc/@{<-}[uurr] &&&
}
\end{array}
$$

\noindent Now, if in the superquiver $Q$, we take vertices $1, 2, 4$ as odd vertices and vertices $3, 5$ as even vertices, then our superquiver $Q$ satisfies the hypothesis (C2) and under our definition of even and odd mutations we have $\eta_{4}\circ\eta_{2}\circ\eta_{1}(Q)=Q'$. 

Note that swapping black and white vertices physically represents switching the helicity of fundamental particles involved in scattering. This shows that the ``flip" move of planar on-shell diagram (that is, switching helicity) provides a geometric combinatorial model for a sequence of odd mutations.

\bigskip

\section{Examples of Cluster Superalgebras}

\bigskip

\subsection{Symplectic-orthogonal supergroup}

\bigskip

\noindent We give a brief introduction on supergroup here, for more details please see Manin's books \cite[Chapter 2]{Manin2} and \cite{Manin} for more details. 

A supermatrix over a supercommutative superalgebra $\mathcal A$ is a matrix $M=\left(\begin{array}{cc}
A & B\\
C & D
\end{array}\right)$, where the matrices $A,D$ have even entries and they are of sizes $m\times m$ and $n\times n$, respectively. The matrices $B,C$ have odd entries
and are of sizes $m\times n$ , $n\times m$, respectively.

The general linear supergroup $GL(m|n)$ is the group of all invertible supermatrices $M$ of size $(m+n)\times (m+n)$, i.e. supermatrices $M$ such that the superdeterminant 
$$sdet(M)=det(A-BD^{-1}C)det(D^{-1})$$ 
is invertible in $\mathcal A$, or equivalently,  if both $A$ and $D$ are invertible in the usual sense.  Superdeterminant is also known as Berizinian \cite{Manin}. 
Let $A^{t}$ denote transpose 
of a matrix $A$ and $M^{st}$ denote the supertranspose of a supermatrix M. Then 

$$M^{st}=\left(\begin{array}{cc}
A & B\\
C & D
\end{array}\right)^{st}=\left(\begin{array}{cc}
A^{t} & C^{t}\\
-B^{t} & D^{t}
\end{array}\right)$$

Let us write 
$$J_{m}=\left(\begin{array}{cc}
0 & I_{m}\\
-I_{m} & 0
\end{array}\right),
\quad
K_{2k+1}=\left(\begin{array}{cccc}
-1 &0&0\\
0&0&-I_k\\
0&-I_k&0
\end{array}\right),
\quad
K_{2k}=\left(\begin{array}{cccc}
0&-I_k\\
-I_k&0
\end{array}\right)
$$
and 
$$J_{m,n}=diag(J_{m}, K_{n}).$$

\medskip

\noindent The symplectic-orthogonal supergroup $SpO(2m|n)$ is defined via its functor of points. Let $A=A_0\oplus A_1$ be a supercommutative superalgebra. The set of $A$-points of the symplectic-orthogonal supergroup $SpO(2m|n)$ consists of $(2m+n)\times (2m+n)$ supermatrices $M$ with entries in $A$ such that $sdet(M)=1$ and $M^{st} J_{m,n} M=J_{m,n}$. Note that this supergroup is denoted as $PC(2|1)=SC(2|1)=SpO(2|1)_0$ in Manin's book \cite[p37]{Manin2}).

\bigskip
We have the following analogue of a result by Ovsienko \cite[v1, Example 4.6]{Ovsienko}. 
\begin{thm} \label{osp1}
The symplectic-orthogonal supergroup $SpO(2|1)$ over any supercommutative superalgebra $A=A_0\oplus A_1$ admits a cluster superalgebra structure.
\end{thm} 

\begin{proof}
Let $A=A_{0}\oplus A_{1}$ be a supercommutative superalgebra. The set of $A$-points of the symplectic-orthogonal supergroup $SpO(2|1)$ consists of supermatrices $\left(\begin{array}{ccc}
a & b & \gamma\\
c & d & \delta\\
\alpha & \beta & e
\end{array}\right)$ satisfying the following identities\footnote{Here is the detail: 
$M^{st} J_{1,1} M=J_{1,1}$
$\Longrightarrow$
$
\begin{pmatrix}
a & c & \alpha\\
b & d & \beta\\
-\gamma & -\delta & e
\end{pmatrix}
\begin{pmatrix}
0&1&0\\
-1&0&0\\
0&0&-1
\end{pmatrix}
\begin{pmatrix}
a & b & \gamma\\
c & d & \delta\\
\alpha & \beta & e
\end{pmatrix}
=\begin{pmatrix}
0&-bc+ad-\alpha\beta&-c\gamma+a\delta-\alpha e\\
-ad+bc-\beta\alpha&0&-d\gamma+b\delta-\beta e\\
a\delta-c\gamma-\alpha e&b\delta-\alpha\gamma-e\beta &2\delta\gamma-e^2
\end{pmatrix}
=\begin{pmatrix}
0&1&0\\
-1&0&0\\
0&0&-1
\end{pmatrix}
$
$\Rightarrow$
\textcircled{1} $ad-bc-\alpha\beta=1$, \textcircled{2} $\alpha e-a\delta+c\gamma=0$, \textcircled{3} $\beta e-b\delta+d\gamma=0$, \textcircled{4} $2\delta \gamma-e^2=-1$. 

From \textcircled{1} we have the first identity of \eqref{OSpRule21}. 

Now we prove the second identity. From \textcircled{2}\textcircled{3} we get \textcircled{5}$\alpha\beta e^2=(a\delta-c\gamma)(b\delta-d\gamma)=(ad-bc)\gamma\delta$; multiplying both sides by $\gamma\delta$, we get \textcircled{6} $\alpha\beta\gamma\delta e^2=0$.
Rewrite \textcircled{4} in the form \textcircled{7} $e^2=1-2\gamma\delta$. Substituting this in \textcircled{6}, we have
$\alpha\beta\gamma\delta(1-2\gamma\delta)=\alpha\beta\gamma\delta=0$. Therefore we can rewrite \textcircled{5} as
$\alpha\beta(e^2+2\gamma\delta)=(ad-bc-\alpha\beta)\gamma\delta$. By \textcircled{1}\textcircled{6}, the above reduces to \textcircled{8}  $\alpha\beta=\gamma\delta$. 
Meanwhile, $1=sdet(M)=det(\begin{pmatrix} a&b\\c&d\end{pmatrix}-e^{-1}\begin{pmatrix}\gamma\alpha&\gamma\beta\\ \delta\alpha&\delta\beta\end{pmatrix})e^{-1} 
= (ad-bc+e^{-1}(d\alpha\gamma+a\beta\delta-b\alpha\sigma-c\beta\gamma))e^{-1}
=(1+x)e^{-1}$, where
$x= \alpha\beta+e^{-1}(d\alpha\gamma+a\beta\delta-b\alpha\sigma-c\beta\gamma)$ and it is easy to check that $x^2=0$. Thus $e=1+x$, $e^2=(1+x)^2=1+2x$. Comparing with \textcircled{7}, we get $x=-\gamma\delta=-\alpha\beta=\beta\alpha$, therefore $e=1+x=1+\beta\alpha$, and the second identity of \eqref{OSpRule21} is proved. 

To show the third and fourth identities, substituting $e=1+\beta\alpha$ in 
\textcircled{2} and \textcircled{3}, we get $\alpha=a\delta-c\gamma$, $\beta=b\delta-d\gamma$. Then $b\alpha-a\beta=b(a\delta-c\gamma)-a(b\delta-d\gamma)=(ad-bc)\gamma=(1+\alpha\beta)\gamma=(1+\gamma\delta)\gamma=\gamma$, and $d\alpha-c\beta=d(a\delta-c\gamma)-c(b\delta-d\gamma)=(ad-bc)\delta=(1+\alpha\beta)\delta=(1+\gamma\delta)\delta=\delta$. 
}:

\begin{equation} \label{OSpRule21}
ad=1+bc+\alpha \beta, e=1+\beta\alpha, \gamma = b \alpha-a \beta  , \delta=d \alpha-c \beta . 
\end{equation}

The elements $a, b, c, d, e \in A_{0}$ and $\alpha, \beta, \gamma, \delta \in A_{1}$. Note that the elements $a, b, c, d, \alpha, \beta$ generate $SpO(2|1)$. Interested readers may see \cite{Manin2} for the detailed computation on how to get the above relations.   

Choose $a, b, c, \in A_{0}$ and $\alpha, \beta \in A_{1}$. Consider the initial seed $(X|Y)$ with $X=\{a, b, c\}$ where $b, c$ are frozen and $Y=\{\alpha, \beta\}$ and consider the following superquiver $Q$: 

\bigskip

$$
\begin{array}{ccccccc}
 \xymatrix{
& \alpha \ar@{->}[rd]  &&
 \beta \ar@{<-}[ld]&&\\
{\color{blue} b}\ar@{<-}[rr]&& a\ar@{->}[rr]&& {\color{blue} c} &&
}
\end{array}
$$

\bigskip

\noindent Note that $Q$ satisfies the restriction (C1) but not (C2); the only path of the form $y_i\to x_k\to y_j$ is $\alpha\to a\to\beta$, and $a$ is not adjacent to any mutable even vertex (since $b,c$ are not mutable). We have 
\begin{equation}
\mu_{a}(a)=\frac{1}{a}[bc+1+\alpha\beta]
\end{equation}
where the third term $\alpha\beta$ is the summation in the last formula of \eqref{evenmutation}. Set $\mu_{a}(a)=d$. Then  $ad=1+bc+\alpha\beta$. This gives us the first relation of the equation \eqref{OSpRule21}. Note that $\mu_a^2(a)=a$, so iterating even mutations does not produce more new mutable even variables other than $a$ and $d$, thus the set of even variables is $\{a,b,c,d\}$.

Next, it is clear that every odd variable can be generated by $\{a,b, c, \alpha,\beta\}$. This shows that the symplectic-orthogonal supergroup $SpO(2|1)$ over any supercommutative superalgebra $A=A_0\oplus A_1$ admits a cluster superalgebra structure.
\end{proof}

\noindent We do not know whether the more general result that the symplectic-orthogonal supergroup 
$SpO(2m|n)$ over any supercommutative superalgebra $A=A_0\oplus A_1$ admits a cluster superalgebra structure is true or not. 

\bigskip

\noindent For $m=1$ and $n=2$, we have

\begin{thm} \label{osp2}
The symplectic-orthogonal supergroup $SpO(2|2)$ over any supercommutative superalgebra $A=A_0\oplus A_1$ is quotient of a cluster superalgebra.
\end{thm} 

\begin{proof} 
The symplectic-orthogonal supergroup $SpO(2|2)$ over a supercommutative superalgebra $A=A_0\oplus A_1$ consists of supermatrices

\medskip

$$M=\left(\begin{array}{cccc}
a & b & \gamma_1 &\gamma_2\\
c & d & \delta_1 & \delta_2\\
\alpha_1 & \beta_1 & e_1 & e_2\\
\alpha_2 & \beta_2 & e_3 & e_4
\end{array}\right)$$ such that $sdet(M)=1$ and $$M^{st}J_{1,2}M=J_{1,2}.$$

 After working out the computational details as in the proof of the previous theorem, the above condition gives us the following set of equations 

\bigskip

\begin{equation} \label{OSpRule22}
\begin{array}{l}
ad=1+bc-\alpha_1 \beta_2-\alpha_2\beta_1\\
e_1e_4+e_2e_3=1-\gamma_1\delta_2-\gamma_2\delta_1\\
e_1e_3=-\gamma_1\delta_1\\
e_2e_4=-\gamma_2\delta_2\\
-c\gamma_1+a\delta_1=e_1\alpha_2+e_3\alpha_1\\
-c\gamma_2+a\delta_2=e_2\alpha_2+e_4\alpha_1\\
-d\gamma_1+b\delta_1=e_1\beta_2+e_3\beta_1\\
-d\gamma_2+b\delta_2=e_2\beta_2+e_4\beta_1
\end{array}
\end{equation}

\bigskip

\noindent Choose $a, b, c, e_1, e_2, e_3\in A_{0}$ and $\alpha_1, \alpha_2, \beta_1, \beta_2, \gamma_1, \gamma_2, \delta_1, \delta_2 \in A_{1}$. 

\bigskip

\noindent Consider the initial seed $(X|Y)$ with $X=\{a, b, c, e_1, e_2, e_3\}$ where $b, c, e_2, e_3$ are frozen and $Y=\{\alpha_1, \alpha_2, \beta_1, \beta_2, \gamma_1, \gamma_2, \delta_1, \delta_2\}$ and consider the following superquiver:

\newpage 

$$
\begin{array}{ccccccc}
 \xymatrix{
&& \beta_1\ar@{->}[rdd]\ar@{->}[rr] &&
 \alpha_1\ar@{<-}[ldd]&&\\
& \beta_2\ar@{->}[rrd]\ar@{->}[rrrr]  &&&& \alpha_2\ar@{<-}[lld]\\
{\color{blue} b}\ar@{<-}[rrr]&& &a\ar@{->}[rrr]&&  &{\color{blue} c}&
}
\end{array}
$$

$$
\begin{array}{ccccccc}
 \xymatrix{
{\color{blue}e_2}\ar@{<-}[rrr]\ar@{<-}[rd]&&& e_1\ar@{->}[rrr]\ar@{<-}[lld]\ar@{<-}[ldd]\ar@{->}[rrd]\ar@{->}[rdd]&
&&{\color{blue}e_3}\\
&\delta_2\ar@{->}[rrrr]\ar@(dl,dr)_{\id} &&&& \gamma_2\\
&& \delta_1 \ar@{->}[rr] && \gamma_1&  &&}\\
\end{array}
$$

\noindent Note that $Q$ satisfies the restriction (C1) but not (C2); there are four paths of the form $y_i\to x_k\to y_j$ without arrows from $y_i$ to $y_j$: $\beta_2\to a\to\alpha_1$, $\beta_1\to a\to\alpha_2$, $\delta_2\to e_1\to\gamma_1$, $\delta_1\to e_1\to\gamma_2$,  and $a, e_1$ are not adjacent to any mutable even vertex (since $b,c,e_2,e_3$ are not mutable). 
So mutating along the directions of vertices $a$ and $e_1$, we get
\begin{equation}
\mu_{a}(a)=\frac{1}{a}[bc+1+ \beta_2\alpha_1+\beta_1\alpha_2]
\end{equation}
\begin{equation}
\mu_{e_1}(e_1)=\frac{1}{e_1}[1-e_2e_3+ \delta_2\gamma_1+\delta_1\gamma_2]
\end{equation}
where the last two terms of both come from the summation of the last formula in \eqref{evenmutation}.

Setting $\mu_a(a)=d$ and $\mu_{e_1}(e_1)=e_4$, we get the first two equations of \eqref{OSpRule22}. 

\noindent This shows that the supergroup $SpO(2|2)$ is a quotient of the cluster superalgebra $C_\mathbb{K}(X|Y, Q)$.
\end{proof}

\subsection{Super Grassmannian}

\bigskip

\noindent The super Grassmannian $G(r|s; m|n)$ is the supermanifold of $(r, s)$-dimensional superverctor subspaces $U$ of a given $(m, n)$-dimensional supervector space $V=V_0\oplus V_1$, that is, $U\subset V$, $dim(V_0)=m$, $dim(V_1)=n$, $dim(U\cap V_0)=r$ and $dim(U\cap V_1)=s$. 

The supervariety of super Grassmannian is, in general, not a projective supervariety, contrary to the classical setting. However, the particular supervariety of super Grassmannian $G(2|0; 4|1)$ of $(2|0)$ planes inside the superspace $\mathbb C^{4|1}$ is a projective supervariety. This is the only example we could find in literature where the supersymmetric analogue of Pl\"{u}cker embedding and the complete description of the coordinate superalgebra $O(G(2|0; 4|1))$ given in terms of generators and relations is available. Interested readers may see the original source \cite{CFL} for more details but for the sake of completeness we give relevant details from \cite{CFL} below.

We will denote the super Grassmannian $G(2|0; 4|1)$ with $Gr$ and realize it via its functor of points, say $h$. For a generic superalgebra $A,$ the $A$-points of $Gr$ consists of the projective modules of rank $(2|0)$ in $A^{4|1}:=A\otimes \mathbb{C}^{4|1}.$ However, we may consider $A$ to be a local superalgebra. For a local superalgebra $A,$ $h_{Gr}(A)$ consists of free submodules of rank $(2|0)$ in $A^{4|1}$ as projective modules over a local superalgebra are free. Consider the canonical basis $\{a_1,a_2,a_3,a_4,\gamma_1\}$ for $A^{4|1}$.  

A free submodule of rank $(2|0)$ in $A^{4|1}$ can be specified by independent vectors $\mathbf{u}$ and $\mathbf{v},$ which in the canonical basis $\{a_1,a_2,a_3,a_4,\gamma_1\}$ are given by two columns vectors that span the subspace 
$$U=span\{\mathbf{u},\mathbf{v}\}$$

It turns out that $Gr$ is a projective supervariety with an embedding into projective superspace given by the transformation among the functors 
$$P: h_{Gr}\longrightarrow h_{P(E)}$$
where $E$ is the superspace $$E=\bigwedge\nolimits^2\mathbb{C}^{4|1}\approx \mathbb{C}^{7|4}.$$
We have a basis for $E$ given as 
$$\{a_1\wedge a_2, a_1\wedge a_3, a_1\wedge a_4, a_2\wedge a_3, a_2\wedge a_4, a_3\wedge a_4, \gamma_1\wedge \gamma_1, a_1\wedge \gamma_1, a_2\wedge \gamma_1, a_3\wedge \gamma_1, a_4\wedge \gamma_1\}$$

$$h_{Gr}(A)\xlongrightarrow{\text \ \ \ P_A\ \ \ } h_{P(E)}(A)$$
$$span_A\{\mathbf{u},\mathbf{v}\}\longmapsto span_A\{\mathbf{u}\wedge \mathbf{v}\}$$

Note that if $U$ is a free module of rank $(2|0),$ then $\bigwedge\nolimits^2 U$ is a free module of rank $1|0,$ and so it is an element of $h_{P(E)}(A).$

The image of $P_A(h_{Gr}(A))$ is the subset of even elements in $h_{P(E)}(A)$ decomposable in terms of two even vectors of $A^{4|1}$. Consider an even element $P\in h_{P(E)}(A)$. Then we have $P=p+\lambda \wedge \gamma_1 +a_{11} \gamma_1 \wedge \gamma_1$, with $p=p_{12} a_1\wedge a_2 + \cdots + p_{34} a_3 \wedge a_4$ and $\lambda=\lambda_1 a_1 + \cdots + \lambda_4 a_4$ with $p_{ij}, a_{11} \in A_0$ and $\lambda_i\in A_1$. Clearly $P$ is decomposable if and only if $P=(r+ \alpha_1 \gamma_1) \wedge (s+ \beta_1 \gamma_1)$ with $r=r_1 a_1 + \cdots + r_4a_4$ and $s=s_1a_1 + \cdots + s_4a_4$ where $r_i, s_i \in A_0$ and $\alpha_1, \beta_1\in A_1$. This gives us $p\wedge p=0$, $p\wedge \lambda=0$, $\lambda \wedge \lambda=2a_{11}p$ and $\lambda a_{11}=0$. After the simplifications, we have the following super Pl\"{u}cker relations; 

$$p_{12}p_{34}-p_{13}p_{24}+p_{14}p_{23}=0$$
$$p_{ij}\lambda_k-p_{ik}\lambda_j+p_{jk}\lambda_i=0, \text{\ } 1\leq i<j<k\leq 4$$
$$\lambda_i\lambda_j=a_{11}p_{ij}, \text{\ } 1\leq i<j\leq 4$$
$$\lambda_i a_{11}=0$$

\noindent Consider the initial seed $(X|Y)$, where $X=\{p_{12}, p_{14}, p_{23}, p_{24}, p_{34}, a_{11}\}$ with $p_{24}$ being mutable and others frozen; $Y=\{\lambda_1, \lambda_2, \lambda_4\}$ with $\lambda_2$ being mutable and others frozen. 

Consider the following superquiver $Q$:
$$
\begin{array}{ccccccc}
\xymatrix{
& \color{blue}{p_{23}}\ar@{->}[rd]  &&
\color{blue}{p_{34}}\ar@{<-}[ld]\ar@{->}[r]& \color{blue}{a_{11}} & \\
\color{blue}{p_{14}}\ar@{->}[rr]\ar@/^/[d]&& p_{24}\ar@{->}[rr]\ar@{->}[lld]&& \color{blue}{p_{12}}\ar@{->}[d] \ar@/_1pc/[llll]&&\\
\lambda_{2}\ar@{->}[rrrr]\ar@/^/[u]\ar@{->}[rrrru]&& && \lambda_{4} &&\\
& \lambda_{1}\ar@{->}[lu] \ar@{->}[ruu]&& &&\\
}
\end{array}
$$
Note that $Q$ vacuously satisfies the restriction (C1) and (C2) because there are no paths of the form $y_i\to x_k\to y_j$ with $x_k$ mutable and without arrows from $y_i$ to $y_j$.

Mutating at even vertex $p_{24}$, we get $$\mu_{p_{24}}(p_{24})=\frac{1}{p_{24}}[p_{12}p_{34}+p_{23}p_{14}].$$ Setting $\mu_{p_{24}}(p_{24})=p_{13}$, we get $$p_{12}p_{34}-p_{13}p_{24}+p_{14}p_{23}=0.$$ Mutating at odd vertex $\lambda_2$, we get $$\eta_2(\lambda_2)=\frac{1}{p_{14}}[\lambda_4 p_{12}+\lambda_1 p_{24}].$$ Setting $\eta_2(\lambda_2)=\lambda_3$, we get $$p_{12}\lambda_4-p_{14}\lambda_3+p_{24}\lambda_1=0.$$ Thus, we have

\begin{thm}
The coordinate superalgebra of the super Grassmannian $G(2|0; 4|1)$ is the quotient of the cluster superalgebra $C_{\mathbb K}(X|Y, Q)$. More precisely, the coordinate superalgebra of the super Grassmannian $G(2|0; 4|1)$ is $C_{\mathbb K}(X|Y, Q)/\mathcal{I}$ where 
\begin{eqnarray*}
	\mathcal{I}&=& < \lambda_1\lambda_2-a_{11}p_{12},\lambda_1\lambda_3-a_{11}p_{13},\lambda_1\lambda_4-a_{11}p_{14},\\
	&& \lambda_2\lambda_3-a_{11}p_{23},\lambda_2\lambda_4-a_{11}p_{24},\lambda_3\lambda_4-a_{11}p_{34},\\
	&& \lambda_1a_{11}, \lambda_2a_{11}, \lambda_3a_{11}, \lambda_4a_{11} >.
\end{eqnarray*}
\end{thm}

\noindent Note that it is not much of a setback not being able to obtain the relations in $\mathcal{I}$ as these are not in the spirit of mutation in classical cluster algebras.

\bigskip

\section{Finite mutation type superquiver}

\noindent In the theory of cluster algebras, a quiver $Q$ is said to be of {\it finite mutation type} if its mutation-equivalence class is finite. It is known \cite{FST} that a connected quiver $Q$ with at least three vertices is of finite mutation type if and only if it comes from the triangulation of a surface or it is mutation-equivalent to one of the exceptional types $E_6, E_7, E_8, E_6^{(1)}, E_7^{(1)}, E_8^{(1)}, E_6^{(1,1)}, E_7^{(1,1)}, E_8^{(1,1)}, X_6, X_7$. 

We adapt the notion of mutation-equivalence for classical cluster algebras in the case of cluster superalgebras as follows:

\begin{dfn}
A superquiver $Q'$ is said to be {\it mutation-equivalent} to another superquiver $Q$ if there exists mutations $\sigma_1,\ldots, \sigma_r$ with each $\sigma_i$ being either an even mutation or an odd mutation such that $\sigma_r\circ\cdots\circ\sigma_1(Q)=Q'.$ As in the case of classical cluster algebras, we will denote by $Mut(Q)$ the set of all superquivers mutation-equivalent to a superquiver $Q.$
\end{dfn}

\begin{dfn}
A superquiver $Q$ is said to be of finite mutation type if the number of superquivers $Q'$ that are mutation-equivalent to $Q$ is finite.
\end{dfn}

In the theorem below, we characterize superquivers that are of finite mutation type. 

\begin{thm}
	Let $Q$ be a superquiver and denote by $Q_X$ the full subquiver of $Q$ obtained by removing all odd vertices of $Q$ and let $Q_Y$ be the full subquiver of $Q$ obtained by removing all even vertices of $Q.$ Then $Q$ is of finite mutation type if and only if $Q_X$ and $Q_Y$ are of finite mutation type in the classical sense.
\end{thm}

\begin{proof}
	First note that on $Q_X$ and $Q_Y,$ our notion of superquiver mutation coincides with the classical quiver mutation. Clearly $|Mut(Q)|\geq |Mut(Q_X)|\cdot |Mut(Q_Y)|$  because the natural map 
$Mut(Q)\to Mut(Q_X)\times Mut(Q_Y)$
is surjective: for any $Q_1=\mu_{i_r}\dots \mu_{i_1}(Q_X)$, $Q_2=\eta_{j_s}\dots \eta_{j_1}(Q_Y)$, the quiver $\mu_{i_r}\dots \mu_{i_1}\eta_{j_s}\dots \eta_{j_1}(Q)$ in $Mut(Q)$ restricts to $Q_1$ in $Mut(Q_X)$ and restricts to $Q_2$ in $Mut(Q_Y)$). So if $Q$ is of finite mutation type then $Q_X$ and $Q_Y$ are of finite mutation type. Now suppose $Q_X$ and $Q_Y$ are of finite mutation type. Let $|X|=m,$ $|Y|=n,$ $|Mut(Q_X)|=r,$ and $|Mut(Q_Y)|=s.$ In $Q,$ each vertex in $Q_Y$ can have an arrow between at most $m$ vertices in $Q_X.$ There are four possible cases of arrows between a vertex $x_k$ in $Q_X$ and a vertex $y_j$ in $Q_Y$ in $Q:$
	\begin{enumerate}
		\item There are no arrows between $x_k$ and $y_j.$
		\item There are an equal number of arrows $x_k\to y_j$ as there are $y_j\to x_k.$
		\item There are $a$ number of arrows $x_k\to y_j$ and $b$ number of arrows $y_j\to x_k$ with $a>b.$
		\item There are $a$ number of arrows $x_k\to y_j$ and $b$ number of arrows $y_j\to x_k$ with $a<b.$
	\end{enumerate}
	Applying mutation to an odd vertex satisfying (1) maintains condition (1) on that odd vertex and applying mutation to an odd vertex satisfying (2) maintains condition (2) on that odd vertex. Applying mutation to an odd vertex satisfying (3) changes it to satisfying condition (4) on that odd vertex with $a$ becoming $b$ and $b$ becoming $a.$ Applying mutation to an odd vertex satisfying (4) changes it to satisfying condition (3) on that odd vertex with $a$ becoming $b$ and $b$ becoming $a.$ Thus $(1)-(4)$ encompass all possibilities of connections between an even and odd vertex after applying mutation. Since for each of the $r$ possible quivers in $Mut(Q_x)$ we can have $s$ possible orientations for $Q_Y$ and there are $4$ possible ways an even and odd vertex can have arrows between then and each of the $m$ even vertices can have arrows between at most $n$ odd vertices, we have that $|Mut(Q)|< rs\cdot 4^{mn}.$ Thus $Q$ is of finite mutation type.
\end{proof}

\noindent Note the strict inequality due to it being impossible for a connection between a vertex in $Q_X$ and a vertex in $Q_Y$ to attain every state $(1)-(4)$ through mutation. As previously discussed, the only time a connection between a vertex in $Q_X$ and a vertex in $Q_Y$ can change is when mutating at an odd vertex satisfying $(3)$ or $(4).$ Thus if $Q_X$ and $Q_Y$ are finite mutation type, we can refine the bound on the size of the mutation equivalence class as $|Mut(Q)|\leq rs\cdot 2^{n}$ since for each of the $n$ odd variables $y_j$ there are two possible configurations for the set of arrows between $y_j$ and even variables.

\bigskip

\begin{dfn}
A cluster superalgebra $\mathcal C_\mathbb{K}(X|Y, Q)$ is said to be of finite type if the number of supercluster variables is finite.
\end{dfn}

\bigskip

\noindent Cluster algebras of finite type are completely characterized in \cite{FZ2} and their classification, quite surprisingly, turns out to be identical to the Cartan-Killing classification of semisimple Lie algebras and finite root systems. It is known that a cluster algebra $\mathcal A(X, Q)$ is of finite type if and only if $Q$ is mutation equivalent to Dynkin diagram of type $A$, $D$ or $E$.

\bigskip

\noindent We propose the following problem for further development of the notion of cluster superalgebras.  

\begin{prob}
Characterize cluster superalgebras that are of finite type. 
\end{prob}

\bigskip

\section{Limitations and extensions}\label{section:limitations and extensions}

\bigskip

\noindent  One of the limitations of our study is that we have considered only those superquivers in this paper that satisfy at least one of the conditions (C1) and (C2). The reason behind imposing this restriction is that if we consider an arbitrary superquiver and follow our even and odd mutations, then the Laurent phenomenon fails to hold. We illustrate this in the example below where conditions (C1) and (C2) are both violated. Consider the initial seed $(X|Y, Q)$ with $X=\{x_1, x_2\}$, $Y=\{y_1, y_2\}$ with $x_1$ and $x_2$ mutable and superquiver $Q$ given as:

\begin{equation}
\begin{array}{ccc}
  \xymatrix{ 
 x_1\ar@{->}[rr]&&
x_2\ar@<0pt>@{<-}[lld]\ar@<-3pt>@{->}[d]\\
 y_1&& y_2
}
\end{array}
\end{equation}

It may be checked that  \[\mu_1\circ\mu_2\circ \mu_1(x_1)=\frac{1+x_2+x_1(1+y_1y_2)+x_1x_2}{x_2(1+x_2)}\] 

Clearly $\mu_1\circ\mu_2\circ \mu_1(x_1)$ is not a Laurent polynomial in initial cluster variables $x_1, x_2, y_1, y_2$ and this shows that the Laurent phenomenon fails to hold in the case.

To be able to extend this study to any superquiver, one needs some kind of control over sequence of even mutations. For the rest of the paper we are going to consider arbitrary superquivers but we put a restriction on the sequence of even mutations as follows:

\smallskip

{(C3)} To generate even variables, we allow only those sequence of even mutations where no two mutations are in the same direction.

\smallskip

Under condition (C3), our definition of cluster superalgebras clearly extends to arbitrary superquivers and the Laurent phenomenon holds in this case. As a consequence of this extension, we are able to show that the supercommutative superalgebra generated by all the entries of a superfrieze is a subalgebra of a cluster superalgebra. 
\noindent \begin{dfn} (see  Morier-Genoud, Ovsienko, and Tabachnikov \cite[\S2.2]{MOT} or Ovsienko \cite[\S5.2]{Ovsienko})
A superfrieze of width $n$ over a superring $\mathcal R=\mathcal R_{0} \oplus \mathcal R_{1}$ is an array of elements in $\mathcal R$ arranged as follows:

{\small 
\begin{equation}\label{fig:superfrieze}
\begin{array}{ccccccccccccccccccccccccc}
&\ldots&0&&&&0&&&&0\\[10pt]
\ldots& 0 && 0 && 0
&& 0 && 0 &&\ldots\\[10pt]
\;\;\;1&&&&1&&&&1&&&\ldots\\[10pt]
&\beta_{0,0}&& \boxed{\beta_{\frac{1}{2},\frac{1}{2}}} && \beta_{1,1}
&& \beta_{\frac{3}{2},\frac{3}{2}}&& \beta_{2,2}&&\ldots\\[12pt]
&&\boxed{a_{0,0}}&&&&a_{1,1}&&&&a_{2,2}\\[10pt]
&\beta_{-\frac{1}{2},\frac{1}{2}}&& \beta_{0,1}
&& \boxed{\beta_{\frac{1}{2},\frac{3}{2}}}
&&\beta_{1,2}&& \beta_{\frac{3}{2},\frac{5}{2}}&&\ldots\\[10pt]
a_{-1,0}&&&&\boxed{a_{0,1}}&&&&a_{1,2}&&\\[4pt]
&\iddots&&\iddots&& \ddots&&\ddots&& \ddots&&\!\!\!\ddots\\[4pt]
&&a_{2-n,1}&&&&\boxed{a_{0,n-1}}&&&&a_{1,n}&&&&\\[10pt]
\ldots& \beta_{\frac{3}{2}-n,\frac{3}{2}}&& \beta_{2-n,2}&&\ldots
&& \beta_{0,n}&& \boxed{\beta_{\frac{1}{2},n+\frac{1}{2}}}&& \beta_{1,n+1}\\[10pt]
\;\;\;1&&&&1&&&&1&&&&&\\[10pt]
\ldots&0 && 0 && 0
&& 0 && 0 &&0 &\\[10pt]
&\ldots&0&&&&0&&&&0&\ldots
\end{array}
\end{equation}}
such that, $a_{i,j}\in\mathcal{R}_0$, $\beta_{i,j}\in\mathcal{R}_1$, and elements in each diamond shape
$$
\begin{array}{ccccc}
&&B&&\\[4pt]
&\alpha && \beta &\\[4pt]
A&&&&D\\[4pt]
&\gamma &&\delta &\\[4pt]
&&C&&
\end{array}
$$
satisfy 
\begin{equation}
\label{Rule}
\begin{array}{rcl}
AD-BC&=&1+\delta \alpha,\\[4pt]
B\gamma-A\beta&=&\alpha,\\[4pt]
B\delta-D\alpha&=&\beta.
\end{array}
\end{equation}
\end{dfn}

Consider the following extension of the standard Dynkin quiver of type $A_n$ with initial even variables $X=\{x_1,\ldots, x_n\}$ and odd variables $Y=\{y_1,\ldots, y_{n+1}\}:$

\begin{equation}\label{superan}
	\xymatrix{ 
		 y_1\ar@<0pt>@{<-}[rd]&&
		 y_2\ar@<0pt>@{->}[ld]\ar@<0pt>@{<-}[rd]&&
		 y_3\ar@<0pt>@{->}[ld]\ar@<0pt>@{<-}[rd]&
		\cdots &
		y_n \ar@<0pt>@{<-}[rd]&&
	     y_{n+1}\ar@<0pt>@{->}[ld]&&\\
		& x_1\ar@<0pt>@{->}[rr] && x_2\ar@<0pt>@{->}[rr] && x_3 & \cdots & x_n &
	}
\end{equation}
Let $\widetilde{Q}$ denote the superquiver in (\ref{superan}) and consider $\mathcal C_{\mathbb K}(X|Y, \widetilde{Q})$ with $\mathbb K$ a field of characteristic different from 2. Recall that we assume condition (C3), that is, we do not allow a sequence of even mutations containing two mutations in the same direction.

We have the following result, whose statement and first step of the proof are similar to Ovsienko's \cite[Theorem 3]{Ovsienko}. For readers' convenience, we include an example after the theorem.

\begin{thm}
Consider a generic superfrieze of width $n$ such that the even elements $a_{0,i-1}=x_{i}$ ($1\le i\le n$) and odd elements $\beta_{\frac{1}{2}, \frac{1}{2}+i-1}=y_{i}$ ($1\le i\le n+1$) are free variables (which are in boxes in \eqref{fig:superfrieze}) and every other entries in the superfriezes are written as (rational) functions of them. Then the supercommutative superalgebra generated by all the entries of the superfrieze is a subalgebra of the cluster superalgebra $\mathcal C_{\mathbb K}(X|Y, \widetilde{Q})$, where $X=\{x_1,\dots,x_n\}$ and $Y=\{y_1,\dots,y_{n+1}\}$. Moreover for $i=1,\dots,n$, the element $a_{1,i}=\mu_i\cdots\mu_1(x_i)$, so is a cluster variable.
\end{thm}

\begin{proof}
Step 1. We first show that  for $i=1,\dots,n$, the element $a_{1,i}=\mu_i\cdots\mu_1(x_i)$. This step is identical to the proof of  Ovsienko's \cite[Theorem 3]{Ovsienko}. For simplicity of notation, denote 
$x_{i}'=a_{1,i}$ ($1\le i\le n$) and $y_i'=\beta_{\frac{3}{2}, \frac{1}{2}+i}$ ($1\le i\le n+1$). For $k=1,\dots,n$, the quiver $\mu_{k-1}\cdots\mu_1\widetilde{Q}$ is 
\begin{equation*}
\xymatrix{ 
	 y_1\ar@<0pt>@{<-}[rd]&&
	 y_2\ar@<0pt>@{->}[ld]\ar@<0pt>@{<-}[rd]&&
	 y_3\ar@<0pt>@{->}[ld]&
	\cdots &
	 y_k\ar@<0pt>@{<-}[rd]\ar@<0pt>@{->}[ld]&&
	 y_{k+1}\ar@<0pt>@{->}[ld]\ar@<0pt>@{<-}[rd]&
	\cdots\\
	& x_1'\ar@<0pt>@{->}[rr] && x_2' & \cdots & x_{k-1}'\ar@<0pt>@{<-}[rr] && x_k\ar@<0pt>@{->}[rr] && x_{k+1}& \cdots
}
\end{equation*}
So $x_k'=\mu_k\mu_{k-1}\cdots\mu_1(x_k)$ satisfies
\begin{eqnarray}
x_kx_k'&=& 1+x_{k+1}x_{k-1}'+y_{k+1}y_k.
\end{eqnarray}
This coincides with the first equality in \eqref{Rule} applied to the obvious diamond in the superfrieze. Thus $a_{1,k}=x_k'$. 
\smallskip

Step 2. By the proof of \cite[Proposition 2.9.1]{MOT}, all the entries of the superfrieze are polynomials in $a_{i,i}$ and $\beta_{i,i}$ ($0\le i\le n+2$) (we do not need to mention $\beta_{i+\frac{1}{2},i+\frac{1}{2}}$ because it is equal to $\beta_{i,i}$ by \cite[Proposition 2.3.1]{MOT}). So it suffices to show that  $a_{i,i}$ and $\beta_{i,i}$ are in  $\mathcal C_{\mathbb K}(X|Y, \widetilde{Q})$ for all $i$. Moreover, by periodicity property \cite[Lemma 2.6.1]{MOT}, $\beta_{i+n+3,i+n+3}=-\beta_{i,i}$, $a_{i+n+3,i+n+3}=a_{i,i}$, so it suffices to restrict to the range $0\le i\le n+2$.

\smallskip

Step 2a. To show $\beta_{i,i}$ are in $\mathcal C_{\mathbb K}(X|Y, \widetilde{Q})$ for $i=0,\dots,n+2$: 

By \eqref{Rule}, 
$$\beta_{0,i}=x_iy_{i+1}-x_{i+1}y_i\textrm{ for $1\le i\le n-1$, and }\beta_{0,0}=y_1, \beta_{0,n}=x_ny_{n+1}-y_n.$$ 

Then by  \cite[\S1.3]{MOT}, 
$$\beta_{i,i}=\frac{\beta_{0,i}-\beta_{0,i-1}}{x_i}=y_{i-1}+y_{i+1}-\frac{(x_{i-1}+x_{i+1})y_i}{x_i}\textrm{ for $2\le i\le n-1$.}$$
To show that it is in $\mathcal C_{\mathbb K}(X|Y, \widetilde{Q})$, we only need to show that $\frac{(x_{i-1}+x_{i+1})y_i}{x_i}$ is in $\mathcal C_{\mathbb K}(X|Y, \widetilde{Q})$. This is true by observing
$$\mu_i(x_i)y_i=\frac{x_{i-1}+x_{i+1}+y_{i+1}y_i}{x_i}y_i=\frac{(x_{i-1}+x_{i+1})y_i}{x_i}.$$
By a similar reasoning or for trivial reasons, the following elements are also in $\mathcal C_{\mathbb K}(X|Y, \widetilde{Q})$:
$\beta_{0,0}=y_1,\; \beta_{1,1}=y_2-\frac{(x_2+1)y_1}{x_1},\; \beta_{n,n}=y_{n-1}+y_{n+1}-\frac{(1+x_{n-1})y_n}{x_n},\; \beta_{n+1,n+1}=-x_ny_{n+1}+y_n, \beta_{n+2,n+2}=y_n$. 
Thus all  $\beta_{j,j}$ are in $\mathcal C_{\mathbb K}(X|Y, \widetilde{Q})$.

\smallskip

Step 2b.  To show $x_{i,i}$ are in $\mathcal C_{\mathbb K}(X|Y, \widetilde{Q})$ for $i=0,\dots,n+2$: 

For $3\le i\le n-1$, 
$$\aligned
a_{i,i}&=\frac{x_{i-1}+x_{i+1}+\beta_{i,i}\beta_{0,i-1}}{x_i}
=\frac{x_{i-1}+x_{i+1}+\big(y_{i-1}+y_{i+1}-\frac{(x_{i-1}+x_{i+1})y_i}{x_i}\big)(x_{i-1}y_i-x_{i}y_{i-1})}{x_i}\\
&=\frac{x_{i-1}+x_{i+1}+x_{i-1}y_{i-1}y_i+ y_{i+1}(x_{i-1}y_i-x_{i}y_{i-1}) + (x_{i-1}+x_{i+1})y_iy_{i-1} }{x_i}\\
&=\frac{x_{i-1}+x_{i+1}-x_{i-1}y_iy_{i-1}+ x_{i-1}y_{i+1}y_i + x_{i-1}y_iy_{i-1}+x_{i+1}y_iy_{i-1} }{x_i}-y_{i+1}y_{i-1}\\
&=\frac{x_{i-1}+x_{i+1}+ x_{i-1}y_{i+1}y_i + x_{i+1}y_iy_{i-1} }{x_i}-y_{i+1}y_{i-1}\\
\endaligned
$$
On the other hand,
$$\aligned
x_{i-1}x'_i&=x_{i-1}\frac{1+x_{i+1}(\frac{1+x_ix'_{i-2}+y_iy_{i-1}}{x_{i-1}}) + y_{i+1}y_i}{x_i}
=\frac{x_{i-1}+x_{i+1} (1+x_ix'_{i-2}+y_iy_{i-1} ) +x_{i-1}y_{i+1}y_i}{x_i}\\
&=\frac{x_{i-1}+x_{i+1} + x_{i+1}y_iy_{i-1} +x_{i-1}y_{i+1}y_i}{x_i}+x_{i+1}x'_{i-2}\\
\endaligned
$$
So we see that
$$a_{i,i}=x_{i-1}x'_i- x_{i+1}x'_{i-2}-y_{i+1}y_{i-1}$$
which is in $\mathcal C_{\mathbb K}(X|Y, \widetilde{Q})$. For $i=2$, and $i=n$, a similar proof works. For $i=0,1,n+1,n+2$, we get cluster variables which are obviously in $\mathcal C_{\mathbb K}(X|Y, \widetilde{Q})$. This completes the proof.
\end{proof}
\begin{example}
We use the width-2 superfrieze (shown in \cite{MOT}) to illustrate the idea of proof, see Table \ref{tab:shear}. For the ease of reading, we have applied a horizontal shearing to the tables. 
\begin{table}[h]
\begin{center}
\small\addtolength{\tabcolsep}{-5pt}
$$
\begin{array}{cccccccccccccccccccc|ccccc}
a_{0,-1}&&&&a_{10}&&&&a_{21}&&&&a_{32}&&&&a_{43}&&&&a_{54}\\
\beta_{00}&&\boxed{\beta_{\frac{1}{2}\frac{1}{2}}}&&\beta_{11}&&\beta_{\frac{3}{2}\frac{3}{2}}&&\beta_{22}&&\beta_{\frac{5}{2}\frac{5}{2}}&&\beta_{33}&&\beta_{\frac{7}{2}\frac{7}{2}}&&\beta_{44}&&\beta_{\frac{9}{2}\frac{9}{2}}&&\beta_{55}\\  

\boxed{a_{00}}&&&&a_{11}&&&&a_{22}&&&&a_{33}&&&&a_{44}&&&&a_{55}\\
\beta_{01}&&\boxed{\beta_{\frac{1}{2}\frac{3}{2}}}&&\beta_{12}&&\beta_{\frac{3}{2}\frac{5}{2}}&&\beta_{23}&&\beta_{\frac{5}{2}\frac{7}{2}}&&\beta_{34}&&\beta_{\frac{7}{2}\frac{9}{2}}&&\beta_{45}&&\beta_{\frac{9}{2}\frac{11}{2}}&&\beta_{56}\\  

\boxed{a_{01}}&&&&a_{12}&&&&a_{23}&&&&a_{34}&&&&a_{45}&&&&a_{56}\\
\beta_{02}&&\boxed{\beta_{\frac{1}{2}\frac{5}{2}}}&&\beta_{13}&&\beta_{\frac{3}{2}\frac{7}{2}}&&\beta_{24}&&\beta_{\frac{5}{2}\frac{9}{2}}&&\beta_{35}&&\beta_{\frac{7}{2}\frac{11}{2}}&&\beta_{46}&&\beta_{\frac{9}{2}\frac{13}{2}}&&\beta_{57}\\  

a_{0,2}&&&&a_{13}&&&&a_{24}&&&&a_{35}&&&&a_{46}&&&&a_{57}\\

\\
\hline\\
\\
1&&&&1&&&&1&&&&1&&&&1&&&&1\\
y_1&&\boxed{y_1}&&y_4&&y_4&&y_5&&y_5&&y_6&&y_6&&y_3&&y_3&&-y_1\\  
\boxed{x_1}&&&&x_3&&&&x_5&&&&x_2&&&&x_4&&&&x_1\\
y_7&&\boxed{y_2}&&y_8&&y_9&&y_{10}&&-y_7&&y_2&&-y_8&&y_9&&-y_{10}&&-y_7\\  
\boxed{x_2}&&&&x_4&&&&x_1&&&&x_3&&&&x_5&&&&x_2\\
-y_6&&\boxed{y_3}&&-y_3&&-y_1&&y_{1}&&-y_4&&y_4&&-y_5&&y_5&&-y_6&&y_6\\  
1&&&&1&&&&1&&&&1&&&&1&&&&1\\
\end{array}
$$
\end{center}
\caption{Superfrieze of width 2. Top table shows the indices; bottom table shows the explicit entries.} \label{tab:shear}
\end{table}

In the table, $x_1,x_2,y_1,y_2,y_3$ are initial variables. The other entries are determined as follows (same as \cite[\S2.4]{MOT} with a couple of typos corrected).

The even variables:
$$x_3=\frac{1+x_2+y_2y_1}{x_1},\;  x_4=\frac{1+x_1+x_2+y_2y_1+x_1y_3y_2}{x_1x_2}, \; x_5=\frac{1+x_1+y_2y_1+x_1y_3y_2+x_2y_1y_3}{x_2}$$

The odd entries:
$$
\aligned
&y_4=y_2-\frac{(1+x_2)y_1}{x_1},\;
y_5=y_1+y_3-\frac{(1+x_1)y_2}{x_2},\;
y_6=y_2-x_2y_3,\;
y_7=x_1y_2-x_2y_1,\\
&y_8=\frac{x_2(1+x_2+y_2y_1)y_3-(1+x_1+x_2)y_2}{x_1x_2},\;
y_9=y_3+\frac{-(1+x_1+x_2)y_1+x_1y_1y_2y_3}{x_1x_2},\\
&y_{10}=y_1-x_1y_3
\endaligned
$$
Now we show all entries are in $\mathcal C_{\mathbb K}(X|Y, \widetilde{Q})$: 

$x_1,x_2,y_1,y_2,y_3,y_6,y_7$ are obviously in $\mathcal C_{\mathbb K}(X|Y, \widetilde{Q})$;

since $\mu_1(x_1)=\frac{1+x_2+y_2y_1}{x_1}$, we have $\mu_1(x_1)y_1=\frac{(1+x_2)y_1}{x_1}$, thus $y_4=y_2-\mu_1(x_1)y_1\in \mathcal C_{\mathbb K}(X|Y, \widetilde{Q})$;

similarly, $\mu_2(x_2)=\frac{1+x_1+y_3y_2}{x_2}$, thus
$y_5=y_1+y_3-\mu_2(x_2)y_2\in \mathcal C_{\mathbb K}(X|Y, \widetilde{Q})$;

$x_3=\mu_1(x_1), x_4=\mu_2\mu_1(x_2)\in \mathcal C_{\mathbb K}(X|Y, \widetilde{Q})$;

$x_5=\frac{x_1+1+y_5(x_1y_2-x_2y_1)}{x_2}=x_1x_4+y_1y_3-1\in \mathcal C_{\mathbb K}(X|Y, \widetilde{Q})$;

This shows that the second and third rows of Table \ref{tab:shear} are in $C_{\mathbb K}(X|Y, \widetilde{Q})$, therefore all entries are in $C_{\mathbb K}(X|Y, \widetilde{Q})$ because they can be written as polynomials of the entries in the second and third rows.
\end{example}

\begin{rmk}
A limitation of our definition of cluster superalgebra is that, since we assume (C3), we cannot obtain all nontrivial even entries in the superfrieze as cluster variables. It will be very desirable to extend our definition of cluster superalgebra to satisfy this property.
\end{rmk}

\noindent We close the paper with the remark that in \cite{PZ1} and \cite{PZ2} the notions of even and odd Ptolemy relations have been proposed for super Teichm\"{u}ller spaces. It would be useful to see how do our even and odd mutations compare with their even and odd Ptolemy relations. Also, in a recent preprint \cite{Ted}, super Pl\"{u}cker embedding has been defined for a more general class of super Grassmannians. It would be interesting to see if one can obtain examples of coordinate superalgebras of some other super Grassmannians apart from $G(2|0; 4|1)$ admitting cluster superalgebra structure.  

\bigskip

\noindent {\bf Acknowledgement.} The authors would like to thank Prof. Vera Serganova for her useful comments. We also thank the referee for providing many valuable suggestions which we believe have significantly improved the quality and readability of the paper.

\bigskip

\bigskip

\bigskip

\bigskip
\end{document}